\documentclass[11pt]{article}
\usepackage[utf8]{inputenc}
\usepackage[margin=1in]{geometry}
\usepackage[affil-it]{authblk}
\usepackage{amsthm}
\usepackage{amsmath}
\usepackage{amssymb}
\usepackage{amsfonts}
\usepackage{xcolor}
\usepackage[numbers]{natbib}
\usepackage{dirtytalk}
\usepackage{mathtools}
\usepackage{esint}
\usepackage{hyperref}
\usepackage{appendix}
\usepackage{titlesec}
\titlelabel{\thetitle.\quad}
\DeclareMathOperator{\Div}{\text{div}}

\DeclareMathOperator{\di}{\mathrm{d}}
\DeclareMathOperator{\E}{\mathbf{E}}
\newcommand{\h}{\ell\hat{n}}
\newcommand{\Th}{T_{\h}}
\newcommand{\hoh}{\hat{h}\otimes\hat{h}}
\newcommand{\non}{\hat{n}\otimes\hat{n}}
\newcommand{\ep}{\varepsilon}

\newtheorem{thm}{Theorem}[section]
\newtheorem{prop}{Proposition}

\newtheorem{lemma}[thm]{Lemma}
\newtheorem{assumption}{Assumption}
\newtheorem{remark}{Remark}
\theoremstyle{definition}
\newtheorem{definition}{Definition}[section]
\numberwithin{equation}{section}
\mathtoolsset{showonlyrefs}
\allowdisplaybreaks
 
\makeatletter
\def\@fnsymbol#1{\ensuremath{\ifcase#1\or \dagger\or \ddagger\or
   \mathsection\or \mathparagraph\or \|\or **\or \dagger\dagger
   \or \ddagger\ddagger \else\@ctrerr\fi}}
    \makeatother
    
\providecommand{\keywords}[1]
{
  \small	
  \textbf{\textit{Key words and phrases---}} #1
}

\title{Sufficient conditions for local scaling laws for stationary martingale solutions to the 3D Navier-Stokes equations}
\author{Stavros Papathanasiou\thanks{stavrosp@umd.edu}}
\affil{Department of Mathematics\\ University of Maryland, College Park}
\date{\vspace{-5ex}}
\begin{document}
\maketitle
\begin{abstract}
The main goal of this paper is to obtain sufficient conditions that allow us to rigorously derive local versions of the 4/5 and 4/3 laws of hydrodynamic turbulence, by which we mean versions of these laws that hold in bounded domains. This is done in the context of stationary martingale solutions of the Navier-Stokes equations driven by an Ornstein-Uhlenbeck process. Specifically, we show that under an assumption of \say{on average} precompactness in $L^3,$ the local structure functions are expressed up to first order in the length scale as nonlinear fluxes, in the vanishing viscosity limit and within an appropriate range of scales. If in addition one assumes local energy equality, this is equivalent to expressing the structure functions in terms of the local dissipation. Our precompactness assumption is also shown to produce stationary martingale solutions of the Euler equations with the same type of forcing in the vanishing viscosity limit.
\end{abstract}
\tableofcontents{}
\section{Introduction}
\subsection{Statement of the main result}
In this article we consider the Navier-Stokes equations on a sufficiently smooth bounded domain $D\subset \mathbb{R}^3$ driven by an Ornstein-Uhlenbeck process.
Specifically, the set of equations we are concerned with are: 
\begin{equation}\tag{NSE}\begin{cases} \partial_t u + u\cdot\nabla u + \nabla p = \nu\Delta u + Z_t, \\ 
\Div u = 0 \text{ on } D, \\
u=0 \text{ on } \partial D,\\
dZ_t = -LZ_tdt + dW_t.\end{cases}\label{eqn:NSE}\end{equation}
The process $(W_t)_{t\geq0}$ is a $Q$-Wiener process of the form:
\begin{equation}W_t = \sum_{k=1}^\infty \sigma_ke_k(x)\beta_k(t),
\end{equation}
where $(\beta_k)_{k=1}^\infty$ is a sequence of independent standard Brownian motions defined on the same stochastic basis, $(\sigma_k)_{k=1}^\infty$ is a sequence of complex numbers satisfying the coloring condition
\begin{equation}
    \sum_{k=1}^\infty |\sigma_k|^2 := \ep <\infty
\end{equation} and $(e_k)_{k=1}^\infty$ is an orthonormal basis of eigenvectors of the Stokes operator on the space 
$$H = \overline{\{u\in (C_c^\infty(D))^3| \Div u = 0\}},$$
the completion being taken with respect to the $L^2$ norm.  The covariance operator $Q$ is defined on $H$ by $$Qe_k = |\sigma_k|^2 e_k.$$
The linear operator $L$ satisfies $$c_1 \left<Av,v\right>\leq \left<Lv,v\right>\leq c_2 \left<Av,v\right>$$ for two positive constants $c_1, c_2,$ where $A=-P\Delta$ is the Stokes operator, $P$ being the Leray projection operator. 
Furthermore, the couple $(L,Q)$ is chosen in such a way that the Ornstein-Uhlenbeck SPDE for $Z_t$ given in \eqref{eqn:NSE} is such that weak solutions are actually strong and unique and possesses a unique invariant measure, see for instance sufficient conditions in \cite[Chapter~5]{DPZ}. These conditions on $L,Q$ relate to constructions given in the \hyperref[Appendix]{Appendix}.

The parameter $\nu>0$ is the kinematic viscosity but may also be interpreted as the inverse Reynolds number. In particular, we may fix a reference scale and a reference velocity and treat the vanishing viscosity and infinite Reynolds number limits as one and the same. 

Notice that we do not follow the standard practice of driving the Navier-Stokes equation by a white in time process, opting instead to drive by the more regular Ornstein-Uhlenbeck process $Z_t.$ Although this is not common in the mathematical literature, driving our flow by a stochastic process which is correlated in time is a practice employed in the physics literature, see for instance \cite{TB} for an application to plasma turbulence. Later, we will also make an assumption on the pressure field which seems more plausible for an external forcing with slightly better time regularity than Brownian motion. However, with minor changes in the calculations that follow and under similar assumptions to the ones we will proceed to make, one is able to adapt our results to the cases of white-in-time (i.e. driving directly by a force of the form of $W_t$ above) or an appropriate \say{OU tower} type forcing (see for instance \cite{ASEM} for a definition). This latter type of forcing ensures higher regularity in time for $Z_t$ as well as for $u$. Note also that the results we obtain here can be directly adapted to the Navier-Stokes equations on the torus $\mathbb{T}^3=(\mathbb{R}/2\pi\mathbb{Z})^3$ with some simplifications.

In the context of the 3D Navier-Stokes equations driven by white in time forcing on a bounded domain, the only notion for which we know how to construct global in time solutions is that of \textit{martingale solutions}. These are the weak in the probabilistic sense analogue of Leray-Hopf solutions. We refer to \cite{FG} for their definition and a proof of their existence. In the same paper, the authors also prove the existence of \textit{stationary} martingale solutions.
This very weak notion of a stationary solution provides a substitute for invariant measures for this kind of system, for which it is not even known whether a Feller Markov semigroup can be defined\footnote{a construction of \textit{Markov selections} however can be found in \cite{FRMarkov}}. The main reason that this notion of solution is the best available option is that the standard a priori estimates only suffice to establish convergence in law for the Galerkin approximations of the Navier-Stokes system, so the Skorokhod embedding theorem must used - hence the stochastic basis is part of the construction. In our case, the notion of solution employed will be that of a martingale solution for the coupled process $(u_t,Z_t).$ We present the relevant definition in the next section.

We sketch the construction of stationary martingale solutions for \eqref{eqn:NSE} in \hyperref[Appendix]{Appendix A}, where we merely adapt the scheme used in \cite{FG} for the white-in-time case. The martingale solutions that we construct and work with will satisfy the estimates:
\begin{equation}\frac{\nu}{T}\mathbf{E}\int_0^T\|\nabla u\|_{L_x^2}^2dt \leq c_1\nu U^2 + c_2U^3 <\infty,\label{stb1} \end{equation} 
\begin{equation}\mathbf{E}\sup_{0\leq t \leq T}\|u\|_{L_x^2}^r \leq \left(\E\|u_0\|_{L^2}^p + \frac{C_p}{\nu}\int_0^T\E \|Z_s\|_{H^{-1}}^p ds\right)\exp(\frac{C_pT}{\nu}) \text{ for any } r\in[1,\infty),\label{stb2}\end{equation}
over any time interval $[0,T],$ where $c_1,c_2$ are positive constants depending on the geometric properties of $D$ and $Z_t$ but independent of $\nu$, while $U$ is defined as:
$$U^2 \equiv \lim_{t\to \infty}\frac{1}{t}\int_0^t\E\|u(\tau)\|_{L^2}^2d\tau.$$ Inequality \eqref{stb1} can be derived following the work in \cite{DF} and for stationary solutions it simplifies to:
\begin{equation}\label{stb12}\nu\mathbf{E}\|\nabla u\|_{L_x^2}^2 dt \leq c_1\nu\E\|u\|_{L^2}^2 + c_2(\E\|u\|_{L^2}^2)^{3/2}.
\end{equation}
The latter estimate will be used to deal with the dissipative terms in the proofs of our main theorems.
The estimates \eqref{stb1} and \eqref{stb2} can be interpolated to obtain that for powers $p,q\in[1,\infty], r\in[1,\infty)$ satisfying
$$\frac{2}{q}+\frac{3}{p} \geq \frac{3}{2} \text{ and } \frac{2}{r}+\frac{3}{p} > \frac{3}{2},$$ the $L_t^qL_x^p$ norm has finite $r$ moment over any finite time interval:
\begin{equation}\mathbf{E}\|u\|_{L_t^q L_x^p}^r < \infty.
\label{stb3}\end{equation} See \cite{45} for a proof. A clear expression with appropriate constants for \eqref{stb3} can be obtained by carefully carrying out the appropriate interpolation procedure, but we omit it since we only need it for fixed values of $\nu,$ while it is also weaker than our assumptions. Note that the upper bounds one obtains in \eqref{stb2} and \eqref{stb3} depend on $\nu.$

Given a vector $h\in\mathbb{R}^3$ and a vector field $u,$ we use the following notation for the increments of $u$:
$$\delta_h u(x) := u(x+h) - u(x).$$
The classical third order structure functions of turbulence for which the $4/3$ and $4/5$ laws are usually demonstrated are defined as follows:
\begin{equation}
s_0(\ell) := \left<|\delta_{\ell\hat{n}}u|^2(\delta_{\ell\hat{n}}u\cdot\hat{n})\right> \text{  and  }
s_{||}(\ell) := \left<\left(\delta_{\ell\hat{n}}u\cdot\hat{n}\right)^3\right>,\notag
\end{equation}
where $\left<\cdot\right>$ denotes some sort of averaging which can for instance be a combination of an ensemble, space, time and solid angle averaging, and each of the various averages can be dropped if the flow satisfies appropriate symmetries. Then, the $4/3$ and $4/5$ laws say that for any length scale $\ell$ in an appropriate \textit{inertial range} of scales we have:
$$s_0(\ell) \sim -\frac{4}{3}\ep\ell,$$
$$s_{||}(\ell) \sim -\frac{4}{5}\ep\ell,$$
where $\ep$ denotes the mean energy dissipation per unit mass. The structure functions are related to energy flux through scale $\ell$ and the negative sign in the $4/3$ and $4/5$ laws indicates a cascade of energy to small scales, see the discussion in \cite{Frisch}. 

The inertial range of scales in which these laws hold consists of scales $\ell$ such that $\ell_D \ll \ell_\nu \ll \ell_I,$ where $\ell_D$ is the \textit{dissipative scale}, at and below which viscous dissipation is dominant and $\ell_I$ is the \textit{integral scale}, at which energy is being input into the system.

The $4/5$ law was first derived by Kolmogorov in \cite{K41b} for statistically stationary, homogeneous and isotropic flows satisfying the fundamental axiom of hydrodynamic turbulence, what is called the $0$-th law or anomalous dissipation, which he previously introduced in \cite{K41a}. We describe anomalous dissipation below. The $4/3$ law was already known to Monin and Yaglom (see \cite{MY}), as was its equivalence to the $4/5$ law. See also the article \cite{Aetal} for a discussion of the relation between the two laws. 

In his original theory of turbulence (known as K41), Kolmogorov postulated also what is known as his \say{similarity hypotheses}, which led to the broader scaling prediction for the structure functions of order $p$:
$$\left<(\delta_{\ell\hat{n}}u)^p\right> \sim C_p(\ep\ell)^{p/3}.$$
It was later argued (and supported by experiments and numerics) that this scaling is generally incorrect. This phenomenon is related to the occurence of rare events in the flow and is known as \textit{intermittency}. Thus one needs to introduce intermittency corrections $\tau(p/3)$ such that $p/3+\tau(p/3)$ is the exponent on the right hand side above. See \cite{Frisch} for a discussion of intermittency, as well as \cite{SK,vdW} for evidence of deviation from the exponent $p/3$. At any rate, the third order exponent (i.e. for $p=3$) and in particular the $4/5$ law are considered to be exact. It is also worth remarking that intermittency is related to local averages of the dissipation of the form $$\frac{\nu}{r^3}\int_{|x-y|<r}|\nabla u(y)|^2dy,$$
see for instance \cite{Ob} and \cite{ISY}. Thus the local structure of a turbulent fluid is closely related to deviations from the predictions of K41.

One may observe that since the structure functions are defined in terms of increments of the velocity field, they only make sense, strictly speaking, if the domain of the flow has no boundary. Therefore, one of the points of this article is to obtain $4/3$ and $4/5$ laws that make sense for flows in a bounded domain as well.

We first define our local structure functions.
\begin{definition} Let $\psi\in C^\infty(D).$ The local third order structure functions (with respect to $\psi$) are defined as:
\begin{equation}S_0(\ell) = \frac{1}{4\pi}\mathbf{E}\int_{S^2}\int\psi|\delta_{\ell \hat{n}}u|^2\delta_{\ell \hat{n}}u\cdot\hat{n}dxdS(\hat{n}),
\end{equation}
\begin{equation}S_{||}(\ell) = \frac{1}{4\pi}\mathbf{E}\int_{S^2}\int\psi(\delta_{\ell\hat{n}}u\cdot\hat{n})^3dxdS(\hat{n}).
\end{equation}
\end{definition}
In the above and in everything that follows, the domain $D$ of integration in $x$ has been suppressed. We proceed to state the assumptions that we will employ. 
\begin{assumption}\label{a1}There exists $C>0$ independent of $\nu$ for which our family of stationary martingale solutions $\{u^\nu\}_{\nu>0}$ satisfies:
\begin{gather}
    \sup_{\nu\in (0,1)}\mathbf{E}\|u^\nu\|_{L^3(D)}^3 \leq C \text{ and}\\
    \lim_{|h|\rightarrow 0} \sup_{\nu \in (0,1)}\mathbf{E}\|\delta_h u^\nu\|_{L^3(D)}^3 = 0.
\end{gather}
\end{assumption}
\begin{assumption}\label{a2}There exists $C>0$ independent of $\nu$ and a choice of pressure fields $\{p^\nu\}_{\nu>0}$ corresponding to the solutions $\{u^\nu\}_{\nu>0}$ such that:
\begin{gather}
    \sup_{\nu\in (0,1)}\mathbf{E}\|p^\nu\|_{L^{3/2}(D)}^{3/2} < C  \text{ and} \label{p1}\\
    \lim_{|h|\rightarrow0}\sup_{\nu\in(0,1)}\mathbf{E}\|\delta_h p^\nu\|_{L^{3/2}(D)}^{3/2} = 0.
\end{gather}
\end{assumption}
\begin{assumption}\label{a3} The family of stationary martingale solutions $\{u^\nu\}_{\nu>0}$ satisfies local energy equality:
    \begin{equation}
        2\nu\mathbf{E}\int\psi|\nabla u^\nu|^2 dx = \nu \mathbf{E}\int|u^\nu|^2\Delta\psi dx + \mathbf{E}\int(|u^\nu|^2 +2p^\nu)(u^\nu\cdot\nabla\psi)dx + 2\mathbf{E}\int \psi u\cdot Z_tdx \label{LEE}\tag{LEE}
    \end{equation}
\end{assumption}
We now state our main results. The first of them relates the local structure functions to nonlinear flux terms plus a local contribution from the noise:
\begin{thm}\label{Thm1}
Let $\{u^\nu\}_{\nu>0}$ be a family of stationary martingale solutions to \eqref{eqn:NSE} satisfying \hyperref[a1]{Assumption 1} and \hyperref[a2]{Assumption 2}. Then, for every $\nu\in(0,1)$ there exists a scale $\ell_\nu$ such that $\ell_\nu\rightarrow 0$ when $\nu\rightarrow 0,$ for which we have:

\begin{equation}
\lim_{\ell_I\rightarrow 0}\limsup_{\nu\rightarrow 0}\sup_{\ell\in[\ell_\nu,\ell_I]}\left|\frac{1}{\ell}S_0(\ell)+\frac{2}{3}\mathbf{E}\int (|u|^2 + 2p)u\cdot\nabla\psi dx + \frac{4}{3}\mathbf{E}\int\psi u\cdot Zdx \right| = 0, \label{L431}
\end{equation}
\begin{equation}\lim_{\ell_I\rightarrow 0}\limsup_{\nu\rightarrow 0}\sup_{\ell\in[\ell_\nu,\ell_I]}\left|\frac{1}{\ell}S_{||}(\ell)+\frac{2}{5}\mathbf{E}\int (|u|^2 + 2p)u\cdot\nabla\psi dx + \frac{4}{5}\mathbf{E}\int\psi u\cdot Zdx \right| = 0 \label{L451}
\end{equation}
\end{thm}
Under the additional assumption of local energy equality, the above theorem reduces to a relation between the local structure functions and the local dissipation:
\begin{thm}\label{Thm2}
Let $\{u^\nu\}_{\nu>0}$ be as in \hyperref[Thm1]{Theorem 1.1} and also satisfying \hyperref[a3]{Assumption 3}. Then for the same $\ell_\nu$ as in \hyperref[Thm1]{Theorem 1.1}, we have:
\begin{equation}\lim_{\ell_I\rightarrow 0}\limsup_{\nu\rightarrow 0}\sup_{\ell\in[\ell_\nu,\ell_I]}\left|\frac{1}{\ell}S_0(\ell)+\frac{4}{3}\nu\mathbf{E}\int\psi|\nabla u|^2 dx\right| = 0,\label{L432}\end{equation}
\begin{equation}\lim_{\ell_I\rightarrow 0}\limsup_{\nu\rightarrow 0}\sup_{\ell\in[\ell_\nu,\ell_I]}\left|\frac{1}{\ell}S_{||}(\ell)+\frac{4}{5}\nu\mathbf{E}\int\psi|\nabla u|^2 dx\right|=0.\label{L452}
\end{equation}
\end{thm}

The scale $\ell_\nu$ is only an upper bound on the dissipative scale, at which the viscous forces are dominant. The scale variable $\ell_I$ which we use in our statements should not be confused with the integral scale of the flow. Its role is to eliminate the influence of the large scale forcing and it corresponds to the variable $\ell$ in equation (6.57) in \cite{Frisch} (see the discussion therein).

We refer to both \eqref{L431} and \eqref{L432} as the \textit{local 4/3 law}, and similarly \eqref{L451} and \eqref{L452} will be referred to as the \textit{local 4/5 law}. The derivations of the above laws on the basis of Assumptions \hyperref[a1]{1},\hyperref[a2]{2} and \hyperref[a3]{3} are presented in sections \hyperref[sec:p43]{3} and \hyperref[sec:45]{4} respectively. Note that in order for \hyperref[Thm1]{Theorem 1.1} to be physically relevant, we would reasonably require a condition guaranteeing the existence of a sequence $\nu_k$ of positive numbers decreasing to $0$ such that
$$\lim_{k\rightarrow\infty}\mathbf{E}\int\psi Z_0\cdot u^{\nu_k} dx >0,$$
which would mean that energy is being input in the system at least at a rate independent of the Reynolds number. Had we driven the Navier-Stokes equations by white-in-time forcing, this would be unnecessary since in that case it follows by Ito's formula that the rate of energy input is given by the trace of the covariance of the noise.

\subsection{Discussion and context of our work}
We present first the notion of solution we will be concerned with.
\begin{definition}A \textit{martingale solution} to \eqref{eqn:NSE} consists of a stochastic basis $(\Omega, \mathcal{F}, \{\mathcal{F}_t\}_{t\in[0,T]},\mathbf{P})$ on which we have a family of independent Brownian motions $\{\beta_k\}_{k=1}^\infty$ and a progressively measurable process $(u,Z_t):[0,T]\times\Omega \rightarrow H\times H,$ whose paths are $\mathbf{P}-$almost surely in
$$\left(C([0,T];D(A^{-\frac{\alpha}{2}}))\cap L^\infty(0,T;H)\cap L^2(0,T;V)\right)^2$$
for some $\alpha>1,$ for which we have:
$$dZ_t = -AZ_tdt + \sum_{k=1}^\infty \sigma_k e_k d\beta_k$$
and such that $\mathbf{P}$-almost surely, for every smooth divergence free vector field $v$ with compact support in $D$ we have:
\begin{align}\left<u(t),v\right> + \int_0^t \left<u(s)\cdot\nabla u(s),v\right>ds + \nu\int_0^t\left<\nabla u(s),\nabla v\right>ds = \left<u(0),v\right>+ \int_0^t\left<Z_s,v\right>ds, \notag\end{align}
where $\left<\cdot,\cdot\right>$ denotes the $L^2$ inner product.\\
A \textit{stationary martingale solution} to \eqref{eqn:NSE} is a martingale solution such that for every $\tau >0$, the paths $\left(u(\cdot+\tau),Z_{\cdot+\tau}\right)$ and $\left(u(\cdot),Z_\cdot\right)$ coincide in law.
\label{D1}
\end{definition}
The reader can consult \cite{FG} to examine exactly how $D(A^s)$ is defined (for positive or negative values of $s$) in terms of series expansions with respect to an eigenbasis of $A$ in $H.$ Note that $Z_t$ is acted upon by $L,$ which has a specific relation to $A$ as mentioned earlier, while by taking the (equivalent) Leray projected version of the equation for $u$ we see that $A$ itself acts on $u$. This is why both $u$ and $Z_t$ can be seen in our context in terms of series expansions with respect to an eigenbasis of $A$.

We will frequently abuse the terminology of the above definition and refer to the process $u$ itself as a martingale solution. According to \hyperref[D1]{Definition 1.2}, in constructing martingale solutions we are allowed to pick a probability space and a $Q$-Wiener process on it, for which there is a process $(u_t,Z_t)$ solving \eqref{eqn:NSE} in the weak sense, even if finding such a solution is not possible in some a priori given probability space. We note that even though the equation for $Z_t$ has a unique invariant measure, as the viscosity $\nu$ varies we obtain solutions $(u_t^\nu,Z_t^\nu)$ where the processes $Z_t^\nu$ are only equal to each other in law. This difference is not significant in what follows and we caution the reader that we will mostly ignore it, omitting the index $\nu$ and denoting all of the $Z_t^\nu$ processes by $Z_t.$

In the context of weak solutions to the deterministic Euler and Navier-Stokes equations, a \textit{defect measure} was introduced in \cite{DR}, which is a distribution depending on the solution $u$ that accounts for loss\footnote{or creation; it is now known that there are weak solutions attaining any energy profile; see \cite{BDLSV}. However, the class of Leray-Hopf solutions is not known to have this pathology.} of kinetic energy not due to viscosity but due to lack of smoothness. The defect measure is a limit of expressions related to $s_0(\ell)$ and a $4/3$ law relating the structure function to the defect (but not to the local viscous dissipation) is derived in the same article assuming the existence of the limit $\lim_{\ell\rightarrow 0}(\ell^{-1}s_0(\ell))$. In a similar direction, \cite{Ey2} later derived scaling laws for various third order structure functions, including a $4/5$ law. These works provide relations between the structure functions and the defect measure at fixed Reynolds number (which may be infinite) and without any reference to the inertial range whereas in this article we study the limit of infinite Reynolds number within an appropriate range of scales, without concerning ourselves with an emerging solution of Euler as $\nu\rightarrow 0$. Other differences are that here we study a randomly driven Navier-Stokes system and that we are dealing with boundaries, complicating things for instance through the presence of terms involving the pressure which cannot be simply estimated by means of the Calderon-Zygmund inequality (see the discussion on \hyperref[a2]{Assumption 2} below). More recent, related but improved results for the Euler equations in bounded domains can be found in \cite{DN} and \cite{BTW}. The common thread between all of these works as well as the present one is carefully performing local calculations with energy fluxes on the basis of the equation, under different assumptions each time.

More recently, it was proven in \cite{45} that the $4/5$ law holds for stationary martingale solutions of the Navier-Stokes system on $\mathbb{T}^3$ with white in time forcing, under a very weak condition which the authors called \say{weak anomalous dissipation}. Classically, one defines the dimensionless energy dissipation rate $\ep_\nu:=\nu L \E\|\nabla u\|_{L^2}^2/U^3$ for a characteristic length scale $L$ and a characteristic velocity $U$. For instance, $L$ is taken to be the diameter of the domain or the integral scale, while $U$ is taken to be $(\E\|u^\nu\|_{L^2}^2)^{1/2}$. Anomalous dissipation is the phenomenon that $\ep_\nu$ remains bounded below by a strictly positive number as $\nu\to0.$ If the kinetic energy (and thus $U$) remains bounded, anomalous dissipation means that $\nu\E\|\nabla u^\nu\|_{L^2}^2$ is also nonvanishing. Weak anomalous dissipation is then stated as a slight weakening of the boundedness of $\E\|u^\nu\|_{L^2}^2$ as $\nu\to 0$:
\begin{definition}
A family $\{u^\nu\}_{\nu>0}$ of stationary martingale solutions to \eqref{eqn:NSE} is said to satisfy \textit{weak anomalous dissipation} if:
\begin{equation}\tag{WAD}
\lim_{\nu\rightarrow 0}\nu\mathbf{E}\|u^\nu\|_{L_x^2}^2 = 0.\label{WAD}\end{equation}
\end{definition}
This is equivalent to the vanishing, in the inviscid limit, of the so-called Taylor microscale, introduced in \cite{Tay}. The derivation of the $4/3$ and $4/5$ laws in \cite{45} uses a relation between the third order structure functions and the two-point correlation matrix of $u$ usually referred to as the von Kármán–Horwath-Monin relation. The terms involving two-point correlations descend from the viscous terms of \eqref{eqn:NSE} and can therefore be treated through \eqref{WAD}, furnishing the results. The same line of reasoning does not seem to work in a manner conducive to the derivation of analogues of the $4/3$ and $4/5$ laws in our case. This is because in the analogue of the von Kármán-Horwath-Monin relation
one obtains multiple terms in which $\nu$ does not appear and which unfortunately can not be treated solely on the basis of \eqref{WAD}.

We now give a discussion of our assumptions. For fixed $\nu,$ 
\hyperref[a1]{Assumption 1} provides nothing new. This follows from \eqref{stb3} for $p = q = r = 3.$
Keeping in mind the Kolmogorov-Riesz characterization of compactness in $L^p$ over Euclidean spaces, \hyperref[a1]{Assumption 1} can be informally described as: \textit{the family $\{u^\nu\}_{\nu\in(0,1)}$ is precompact in $L^3(D)$ \say{on average}.} This viewpoint is followed in the proof of \hyperref[IL]{Proposition 1}. Note that the uniform boundedness of $\mathbf{E}\|u^\nu\|_{L^3(D)}^3$ is already a stronger condition than the uniform boundedness of $\mathbf{E}\|u^\nu\|_{L^2(D)}^2,$ which in turn is clearly stronger than \eqref{WAD}. Therefore, in our proofs below, in places where we only need to use weak anomalous dissipation we will clarify it, in order to be clear about where our first assumption really comes into play. 
\begin{remark}\label{rmk1}
In \cite{45} an example is given of a family (parametrized by $\nu$) of Navier-Stokes systems on $\mathbb{T}^3$, driven by a degenerate white-in-time Gaussian noise, and a family of stationary martingale solutions thereof which does not satisfy \eqref{WAD}. Thus in that case \hyperref[a1]{Assumption 1} is also not valid. To see that a similar phenomenon might still occur in the case of Ornstein-Uhlenbeck external forcing, consider \eqref{eqn:NSE} on $\mathbb{T}^3,$ where the driving noise is $(e^{ix_2},0,0)Z_t,$ where $Z_t$ solves the one dimensional SDE $dZ_t = -Z_tdt + d\beta_t,$ $\beta_t$ being a standard Brownian motion. One might construct a stationary martingale solution of 
$$\partial_tu=\nu\Delta u+Z_t(e^{ix_2},0,0)$$
of the form $u^\nu(t) = g(t) (e^{ix_2},0,0),$ where the statistically stationary process $g(t)$ solves the (random) ODE $g'(t) = -g(t)+Z_t.$ For this family, we have:
$$\nu\E\|u^\nu(t)\|_{L^2}^2 = 2\E Z_tg(t),$$
whereby taking the limit $t\to\infty$ and using elementary stochastic calculus one finds that the right hand side is equal to $\frac{1}{\nu+1}.$ Thus \eqref{WAD}, and by extension \hyperref[a1]{Assumption 1}, is not true. On the other hand, note that we just showed that the energy input is bounded away from $0$ in the vanishing viscosity limit. The fact that in this situation the energy input is nonvanishing as $\nu\to0$ (which is the most interesting case) and \hyperref[a1]{Assumption 1} does not hold is due to the degeneracy and low dimensionality of this particular system and does not preclude both of these being true for more complex Navier-Stokes systems. We merely showcase this example to exhibit that our assumptions may fail in the case of Ornstein-Uhlenbeck forcing in the same way as they might fail for white-in-time forcing.\end{remark}

Note that for the white-in-time forced Burgers equation it was shown in \cite{Eetal} (see also the review article \cite{ES} that the energy of the solution is bounded and the energy dissipation rate is nonvanishing as $\nu\to 0,$ both of which are stronger than \eqref{WAD}.

Any uniform in $\nu$ bound for $\mathbf{E}\|u\|_{B_{p,\infty}^\alpha}$ for some $p\geq 3, \alpha\in(0,1)$ would imply \hyperref[a1]{Assumption 1}. Note also that via the relation between structure functions and Besov norms, numerical and experimental evidence suggests the plausibility of this assumption, see \cite{Benzietal,Graueretal}.

\begin{remark}These ideas are also related to the Onsager phenomenology of turbulence. Even though they pertain to the Euler equations, they are still instructive in the study of the Navier-Stokes equations as well. In \cite{Ons} Onsager conjectured that if a weak solution of the Euler equations belongs to $C^{\alpha}$ for $\alpha\in(1/3,1)$ then it conserves energy, and that for $\alpha\in(0,1/3)$ there are weak solutions in $C^\alpha$ dissipating energy. The first part was proven in \cite{Ey1} for a norm stronger than $\|\cdot \|_{C^\alpha}$ and then settled in \cite{CET} in the sharper and more natural context of Besov spaces $B_{3,\infty}^\alpha$. The second part was settled completely by \cite{Is} (see also the review article \cite{BV} and references to earlier work therein). In \cite{DrEy} it is shown in the context of the deterministic Navier-Stokes equations that under a mild lower bound of the form $\sim\nu^\alpha$ on the mean dissipation rate of $u^\nu$ in the limit $\nu\rightarrow 0,$ the family $\{u^\nu\}_{\nu>0}$ can not be bounded in $B_{3,\infty}^{\sigma_\alpha+\epsilon}$ for $\epsilon >0$ and $\sigma_\alpha = \frac{1+\alpha}{3-\alpha}.$
These do not preclude the potential validity of our discussion since on-average precompactness in $L_x^3$ (which would correspond to uniform boundedness and vanishing of increments of $\{u^\nu\}$ in $L_t^3L_x^3$) is a weaker condition that may still be valid. To our knowledge, however, there are no known non-trivial, uniform in $\nu$ a priori estimates for solutions to \eqref{eqn:NSE} (i.e. ones that do not directly follow from \ref{stb12}).\end{remark}

With respect to \hyperref[a2]{Assumption 2}, its role is purely technical, in that in what follows we will need some bounds on the pressure that are uniform in $\nu.$ These estimates follow from \hyperref[a1]{Assumption 1} in the absence of boundaries: they are a consequence of the well known Calderon-Zygmund estimates. The presence of boundaries complicates things at this point. While for the Euler equations we can formally derive a Neumann problem for the pressure by taking the divergence of the equation for $u,$ in the Navier-Stokes setting such a Neumann problem involves the normal trace of $\Delta u,$ which requires a lot more regularity than what we are assuming here. An argument relying on the maximal regularity for the Stokes semigroup as found for instance in \cite[Chapter~5]{RRS} would provide us with bounds on the pressure, yet we do not know how to improve these to estimates uniform in $\nu$. Nevertheless, it is worth mentioning that in view of the aforementioned maximal regularity arguments we expect such estimates to be more readily obtained for a system driven by a process with better time regularity. This is another reason why we use an Ornstein-Uhlenbeck stochastic body force.

As for \hyperref[a3]{Assumption 3}, we note that it is related to the notion of suitable weak solutions for the deterministic Navier-Stokes equations, see \cite{CKN}. Instead of local energy inequality, here we assume \text{equality} but only for a particular spatial cutoff function $\psi.$ To formally derive this one proceeds as usual, i.e. with a \textit{space-time} dependent cutoff of the form $\psi(x)h(t)$ and uses the stationarity to cancel the boundary (in time) contributions\footnote{In using the phrase \say{boundary contributions} we imply that $h$ should be thought of as a placeholder for a sequence of nonnegative, compactly supported smooth functions approximating the constant function $1$ from below on their common domain.} from the integration of the term $\partial_t(\psi h)|u|^2.$ We also refer to the discussion in \cite{LS} in connection to \cite{Shv} for a sufficient condition for local energy equality up to the first blow-up time.

An interesting feature of our assumptions is that they imply the existence of a stationary martingale solution for the Euler equations, arising as a subsequential inviscid limit of $u^\nu$, for which the local scaling laws also hold. These solutions can be seen as a model for ideal turbulence incorporating the $4/3$ and $4/5$ laws in the local sense.

\begin{prop}\label{IL}Under assumptions \hyperref[a1]{1} and \hyperref[a2]{2}, there exists a stationary martingale solution $u^0$ of the Euler equations with associated pressure $p^0$ such that $u^\nu \to u^0$ in $L_{\omega,t,x}^3$ over any bounded time interval, and $u^\nu(0)\to u^0(0)$ in $L_{\omega,x}^3$. Moreover, we have:
\begin{equation}\lim_{\ell_I\to0}\sup_{\ell\in(0,\ell_I)}\left|\frac{1}{\ell}S_0(\ell)+\frac{2}{3}\E\int(|u^0|^2+2p^0)u^0\cdot\nabla\psi dx + \frac{4}{3}\E\int\psi u^0\cdot Z^0dx\right| = 0,\label{L43E}\end{equation}
\begin{equation}\lim_{\ell_I\rightarrow 0}\sup_{\ell\in(0,\ell_I)}\left|\frac{1}{\ell}S_{||}(\ell)+\frac{2}{5}\mathbf{E}\int (|u^0|^2 + 2p^0)u^0\cdot\nabla\psi dx + \frac{4}{5}\mathbf{E}\int\psi u^0\cdot Z^0dx \right| = 0\label{L45E}\end{equation}
\end{prop}
\begin{proof}We will pass to a subsequence of $\nu\in(0,1)$ multiple times and for convenience we will denote the resulting subsequence simply by $\nu$. One may use Skorokhod's embedding theorem to construct a new stochastic basis denoted by $(\Omega,\mathcal{F},\mathcal{F}_t,\mathbf{P})$ and with expectation $\E,$ as well as processes $u^\nu,p^\nu,Z_t^\nu$ on this stochastic basis such that $(u^\nu,Z_t^\nu)$ is a stationary martingale solution of \eqref{eqn:NSE} satisfying our assumptions, as well as a process $Z_t^0$ such that $Z_t^\nu\to Z_t^0$ in $L_{t,x}^2$ $\mathbf{P}$-almost surely. Then one can follow the proof of the Kolmogorov-Riesz compactness criterion (see e.g. \cite{HH}) to obtain a subsequence such that:
\begin{gather}\lim_{\nu\to 0}\E\int_0^T\|u^\nu(t)-u^0(t)\|_{L_x^3}^3dt = 0,\\
\lim_{\nu\to0}\E\|u^\nu(0)-u^0(0)\|_{L_x^3}^3 = 0,\\
\lim_{\nu\to0}\E\int_0^T\|p^\nu(t)-p^0(t)\|_{L_x^{3/2}}^{3/2}dt = 0,
\end{gather}
as well as, $\mathbf{P}$-almost surely:
\begin{gather}\lim_{\nu\to0}\int_0^T\|u^\nu(t)-u^0(t)\|_{L_x^3}^3dt = 0,\\
\lim_{\nu\to0}\|u^\nu(0)-u^0(0)\|_{L_x^3} = 0,\\
\lim_{\nu\to0}\int_0^T\|p^\nu(t)-p^0(t)\|_{L_x^{3/2}}^{3/2}dt.
\end{gather}
From the $\mathbf{P}$-a.s. convergences above it readily follows that $(u^0,Z^0)$ is a stationary martingale solution of the OU driven Euler system, i.e. $\mathbf{P}$-almost surely, for almost every $T$ and for all smooth divergence fields $v$ with compact support in $D$ we have:
\begin{equation}\left<u^0(T),v\right> + \int_0^T\left<u^0(s)\cdot\nabla u^0(s),v\right>ds = \left<u^0(0),v\right>+\int_0^T\left<Z_s^0,v\right>ds.\end{equation}
From the convergences in expectation in conjunction with \eqref{L431} and \eqref{L451} we obtain \eqref{L43E} and \eqref{L45E}.
Lastly, the continuity of the generalized normal trace operator $\Gamma_N:L_{\Div}^3(D)\to W^{-1/3,3}(\partial D)$ (see e.g. \cite{Sohr}) implies the no-penetration condition for $u^0.$
\end{proof}

\begin{remark}
In the deterministic case for domains with boundary, there have been some recent investigations regarding the problem of the inviscid limit. For instance, a modulus of continuity over a finite range of scales for the second order structure function has been shown to imply the existence of a subsequential inviscid limit in \cite{DN2}. Such a modulus is not necessarily implied by \hyperref[a1]{Assumption 1} as the latter has been stated, but would follow (over the entire range of scales) if a uniform bound in $B_{3,\infty}^\alpha$ was imposed as an assumption on $\{u^\nu\}_{\nu}$ instead.\end{remark}

We finally want to discuss some aspects of our work that we deem significant. First, we believe that a considerably different approach would be necessary in order to treat turbulence \textit{near} the boundary, since what we are doing always happens at a fixed distance away from it independently of the viscosity. Second, the conditions that we have assumed our solutions of \eqref{eqn:NSE} to satisfy seem very hard to verify, not least in light of \hyperref[rmk1]{Remark 1}. However, see \cite{LC} for a proof of a weak anomalous dissipation condition for the problem of passive scalar advection and the earlier mentioned \cite{Eetal} for boundedness of the energy in the inviscid limit for the stochastic Burgers equation. At any rate, provided that one can obtain relations similar to the ones we derive in the next section, we believe that similar steps can be taken to derive local cascade laws for other systems that are expected to exhibit this sort of behavior, for instance one could attempt to adapt the derivations of the scaling laws in \cite{2D} or \cite{LC} to bounded domains in the same way that our paper adapts the results in \cite{45}.

\section{Preliminary results}
We first set some notation that will be used throughout:
\begin{equation}\Gamma(t,h) = \mathbf{E}\int\psi(u\otimes T_h u)dx\label{tensorGamma}\end{equation}
\begin{equation}D(t,h)^k = \mathbf{E}\int\psi(\delta_h u\otimes\delta_hu)\delta_h u^kdx.\label{tensorD}\end{equation}
Here and in what follows, we use $T_hf$ to denote the translation of a function (or vector field) $f$ by $-h,$ i.e. $T_hf(x) = f(x+h).$
In what follows, for any order-$2$ tensor $\eta,$ we denote the $(i,j)$ component by $\eta^{ij}.$ 

We first state some basic facts about the regularity of the quantities defined above. The first result states that time averages of $S_0$ and $S_{||}$ are continuous functions of $\ell$. The second one states that time averages of $\Gamma$ are $C^2$ functions of $h$. 
\begin{lemma}\label{L2.1}
Fix $\nu>0$ and let $u$ be a martingale solution of \eqref{eqn:NSE}. The functions:
$$\ell\mapsto \frac{1}{T}\int_0^TS_0(t,\ell)dt,$$ 
$$\ell\mapsto \frac{1}{T}\int_0^TS_{||}(t,\ell)dt$$
are continuous for $\ell\in(0,\infty)$ and satisfy the bound:
\begin{equation}
\sup_{\ell\in(0,1)}\left|\frac{1}{T}\int_0^TS_0(t,\ell)dt\right| + \sup_{\ell\in(0,1)}\left|\frac{1}{T}\int_0^TS_{||}(t,\ell)dt\right| \lesssim  \mathbf{E}\|u\|_{L_x^3}^3.
\end{equation}
If in addition $u$ is a stationary martingale solution of \eqref{eqn:NSE}, $S_0$ and $S_{||}$ are time-independent, continuous in $\ell$ and satisfy:
\begin{equation}
\lim_{\ell\rightarrow0}S_0(\ell) = \lim_{\ell\rightarrow0}S_{||}(\ell) = 0.
\end{equation}
\end{lemma}

\begin{lemma}\label{L2.2}Let $\nu, u$ be as in \hyperref[L2.1]{Lemma 2.1}. For $i,j = 1,2,3,$ $\Gamma^{ij}$ is uniformly bounded and continuous. Moreover, the time averages $$h\mapsto \frac{1}{T}\int_0^T\Gamma^{ij}(t,h)dt$$ are in $C^2(\mathbb{R}^3).$ If $u$ is stationary, $\Gamma^{ij}(t,h) \equiv \Gamma^{ij}(h) \in C^2(
\mathbb{R}^3).$\end{lemma}
The proofs of the preceding lemmata closely follow those of Lemma 2.7 and Lemma 2.8 respectively in \cite{45} with only minor modifications and are therefore omitted. 

In order to establish local versions of the relevant third-order scaling laws we will give a localized variation of a traditional recipe. Kolmogorov himself in \cite{K41b} derived the $4/5$ law using a relation that according to Frisch's discussion in \cite{Frisch} was described for isotropic flows in \cite{KH} and later generalized to the anisotropic case in \cite{MY}, attributed mainly to Monin. This is a relation between the third order structure matrix and the correlation matrix of $u$, where these are given by \eqref{tensorD} and \eqref{tensorGamma} respectively when we take $\psi \equiv 1.$ Following Frisch, we call this relation the von Kármán-Howarth-Monin relation, in short KHM. The main ingredient of our take on Kolmogorov's recipe is a local version of the KHM relation. This is the content of \hyperref[KHM]{Lemma 2.4}.

Since our localization proceeds by means of multiplying the integrands of the classical structure functions by a smooth cutoff function $\psi\in C^\infty(D)$ of the spatial variable (cf. \eqref{tensorD}), we begin by calculating an expression for fluxes in the local structure matrix:
\begin{lemma}\label{23}Let $u \in L^3(D)$ be a divergence-free vector field. Define $\kappa=d(\text{supp}(\psi),\partial D).$ Then, in the sense of distributions in the $h$ variable for $|h|<\kappa$ we have:
\begin{align}
\sum_k\partial_{h_k}\int\psi(\delta_hu\otimes\delta_hu)\delta_hu^k dx &= 
\sum_k\partial_{h_k}\int\psi\left[(u\otimes T_hu) u^k + (T_hu\otimes u)u^k\right]dx \label{2} \\&-\sum_k\partial_{h_k}\int\psi\left[ (u\otimes T_hu) T_hu^k+(T_hu\otimes u) T_hu^k\right]dx\notag\\ &-\int T_h(u\otimes u)\delta_h u\cdot\nabla\psi dx \notag.
\end{align}\end{lemma}
\begin{proof} Expanding the left hand side of \eqref{2} we get:
\begin{align}\sum_k\partial_{h_k}\int \psi\left[(u\otimes T_h u)u^k + (T_h u \otimes u) u^k -(u\otimes T_h u)T_h u^k - (T_h u\otimes u)T_h u^k\right]dx\notag\\
+\sum_k\partial_{h_k}\int\psi\left[(T_hu\otimes T_h u)\delta_h u^k + (u\otimes u)T_hu^k- (u\otimes u)u^k\right]dx.\notag
\end{align}
One easily sees that:
$$\sum_k\partial_{h_k}\int\psi(u\otimes u)u^k dx = \sum_k\partial_{h_k}\int\psi(u\otimes u)T_h u^k dx = 0,$$ where the second equality comes from the incompressibility of $u.$
Next, notice that for $|h|<\kappa$ we can apply the operator $T_{-h}$ inside the integrand and then compute directly:
\begin{align}\sum_k\partial_{h_k}\int\psi(T_h u \otimes T_h u)\delta_h u^k dx &= -\sum_k\partial_{h_k} \int T_{-h}\psi (u\otimes u)\delta_{-h}u^k dx = \notag\\ \int(u\otimes u)\delta_{-h}u\cdot T_{-h}\nabla\psi dx &= - \int (T_hu\otimes T_hu)\delta_h u\cdot\nabla \psi dx.\notag \end{align} 
Combining these calculations we arrive at the result.\end{proof}
In what follows, we will need the double inner product of two tensors $\eta_1^{ij},\eta_2^{ij},$ which we shall denote by $\eta_1:\eta_2 \equiv \sum_{i,j} \eta_1^{ij}\eta_2^{ij}.$ 
Our local version of the KHM relation is the following:
\begin{lemma}[Local KHM relation]Let $u^\nu$ be a stationary martingale solution to \eqref{eqn:NSE} on the time interval $[0,T]$ as postulated in Assumptions \hyperref[a1]{1} and \hyperref[a2]{2}. Let $\eta(h) = (\eta^{ij}(h))_{i,j=1}^3$ be a smooth order $2$ test function of the form:
$$\eta(h) = \phi(|h|)I + \varphi(|h|)\hoh,$$
where $\phi(\ell),\varphi(\ell)$ are smooth functions compactly supported on $(0,\kappa)$ and $\hat{h} = h/|h|.$ In component notation: $$\eta^{ij}(h) = \phi(|h|)\delta^{ij} + \varphi(|h|)\frac{h_ih_j}{|h|^2}.$$ We then have the following equality:
\begin{gather}\frac{1}{2}\int\partial_{h_k}\eta:D(t,h)^kdh +\frac{1}{2}\mathbf{E}\int\int\eta :T_h(u\otimes u)(\delta_hu\cdot\nabla\psi) dx dh= \label{KS1}\\
2\nu \int \Delta \eta:\Gamma(h) dh -\nu\mathbf{E}\int\int\left(2(\partial_{x_k}\psi\partial_{h_k}\eta):(u\otimes T_h u) \label{KS3}+\Delta\psi \eta:(u\otimes T_hu)\right)dxdh \\
+\mathbf{E}\int\int p\eta:(\nabla\psi\otimes T_h u)dx dh -\mathbf{E}\int\int \psi (\nabla(T_h p)\otimes u):\eta dx dh\label{KS4}\\+ 2\mathbf{E}\int\int\eta:(u \otimes T_h Z_t + Z_t \otimes T_h u)dx dh dt
 + \mathbf{E}\int\int u\cdot\nabla\psi\left(\eta:(u\otimes T_hu)\right)dxdh.\label{KSlast}
\end{gather}
\label{KHM}
\end{lemma}
Here and in what follows, $I$ denotes the identity matrix.
\begin{proof}Note that for any stationary process $X_t,$ one has $$\frac{1}{T}\mathbf{E}\int_0^T \|X_t\|_{L^p}^pdt = \mathbf{E}\|X_t\|_{L^p}^p.$$ Therefore Assumptions \hyperref[a1]{1} and \hyperref[a2]{2} give uniform bounds for $L_{\omega,t,x}^3$ norms of $u$ and $L_{\omega,t,x}^{3/2}$ norms of $p$ and their respective increments over fixed finite intervals. We will first prove an intermediate balance relation with an additional integration over a finite time interval and subsequently use \hyperref[23]{Lemma 2.3} and the stationarity of $u.$ 

In order to prove the KHM relation, our strategy is to smoothen the Navier-Stokes equations to get a similar relation and then pass to the limit. Given that we are considering functions on a bounded domain, convolving with a standard mollifier might be problematic since differential operators and convolution do not simply commute. We thus opt for a slightly different approach which is however still rather direct.

We refer to \cite{EG} for the complete definition of the smoothing operators $K_\delta.$ We merely mention that for appropriate smooth maps $\phi_\delta$ and (sufficiently small) fixed $r>0,$ these operators are given by:
$$K_\delta f(x) = \int_{B(0,1)}\rho(y)f(\phi_\delta(x)+r\delta y)dy,$$
$\rho$ being a $C^\infty$ positive and radially symmetric function with integral $1$ and support in $B(0,1).$ If $f$ is in any $L^p$ space for some $p\in[1,\infty),$ $K_\delta f$ is $C^\infty$ and thus we can apply differential operators to it. The main benefit of the operators $K_\delta$ is that one is able to control their commutators with differential operators in the limit $\delta \rightarrow 0$.

By applying $K_\delta$ to $\eqref{eqn:NSE},$ we obtain:
\begin{equation}
    \partial_t K_\delta u + K_\delta \Div(u\otimes u) + K_\delta \nabla p = \nu K_\delta \Delta u + K_\delta Z_t \iff \end{equation}
    \begin{equation}\partial_t K_\delta u + \Div K_\delta (u\otimes u) + \nabla K_\delta p - \nu \Delta K_\delta u - K_\delta Z_t = -[K_\delta, \Div](u\otimes u) - [K_\delta,\nabla ]p - \nu[K_\delta, \Delta] u.
\end{equation}
The product rule gives:
$$\partial_t(\psi K_\delta u\otimes T_hK_\delta u) = \psi \partial_t K_\delta u \otimes T_h K_\delta u + \psi K_\delta u \otimes \partial_t T_h K_\delta u,$$ therefore:

\begin{gather}
\partial_t\int\psi K_\delta u \otimes T_h K_\delta u dx = \int \psi[K_\delta \partial_tu \otimes T_h K_\delta u + K_\delta u \otimes T_h K_\delta \partial_t u]dx\\
= -\int \psi[\partial_k K_\delta(u^ku)\otimes T_h K_\delta u + K_\delta u \otimes \partial_k T_h K_\delta(u^ku)]dx\\
-\int_{D}\psi[\nabla K_\delta p \otimes T_h K_\delta u + K_\delta u \otimes \nabla T_h K_\delta p]dx
+\nu\int_{D}\psi[\Delta K_\delta u \otimes T_h K_\delta u + K_\delta u \otimes \Delta T_h K_\delta u]dx\\
+\int\psi[K_\delta Z_t \otimes T_h K_\delta u + K_\delta u \otimes T_h K_\delta Z_t]dx +I_C,
\end{gather}
where $I_C$ is the contribution of the commutator terms:
\begin{align}
I_C := -&\int\psi\left([K_\delta, \Div](u\otimes u) + [K_\delta,\nabla ]p + \nu[K_\delta, \Delta] u\right)\otimes T_h K_\delta u dx\\
-&\int\psi K_\delta u \otimes \left([K_\delta, \Div](u\otimes u) + [K_\delta,\nabla ]p + \nu[K_\delta, \Delta] T_h u\right)dx,
\end{align}
and should not be confused with the identity matrix $I.$
The contribution of the nonlinearity becomes:
\begin{equation}\int(K_\delta u \cdot \nabla\psi) K_\delta u\otimes T_h K_\delta u dx 
+ \partial_{h_k}\int\psi K_\delta u \otimes (u^k T_h K_\delta u - T_h K_\delta(u^k u))dx.
\end{equation}
The contribution of the pressure terms:
\begin{equation}\int K_\delta p \nabla\psi \otimes T_h K_\delta u dx + \nabla_h \int \psi K_\delta p T_h K_\delta u dx - \nabla_h^\perp\int \psi K_\delta u T_h K_\delta p dx. 
\end{equation}
The contribution of the viscous terms:
\begin{equation}\int\left(\Delta\psi K_\delta u \otimes T_h K_\delta u dx + 2\psi K_\delta u \otimes (\nabla\psi \cdot \nabla T_h K_\delta u) + 2\psi K_\delta u \otimes T_h \Delta K_\delta u \right)dx.
\end{equation}

The integrations by parts performed above give no boundary terms because $\psi$ is compactly supported inside $D.$
For brevity's sake, from now on we use the abbreviations $u_\delta = K_\delta u, p_\delta = K_\delta p$ and $ \Gamma_\delta(t,h) = \int\psi u_\delta\otimes T_h u_\delta dx.$ 
Upon contracting by the tensor $\eta$ and integrating (by parts) in $h$, in $t$ and in $\omega,$ as well as using the above calculations, we obtain:
\noeqref{KHMdelta2,KHMdelta3,KHMdelta4,KHMdelta6,KHMdelta7}
\begin{align}
&\int \eta:\Gamma_\delta(T,h)dh - \int \eta:\Gamma_\delta(0,h) dh =  \label{KHMdelta1}\\
&- \mathbf{E}\int_0^T \int\int\psi\left(\partial_{h_k}\eta:(u_\delta\otimes T_h u_\delta)\right)(u_\delta^k-T_hu_\delta^k)dx dh dt\label{KHMdelta2}\\
&+\nu \mathbf{E}\int_0^T\int\int \eta:(u_\delta\otimes T_h u_\delta)\Delta\psi dx dh dt-2\nu\mathbf{E}\int_0^T\int\int(\partial_{x_k}\psi\partial_{h_k}\eta):(u_\delta\otimes T_h u_\delta)dx dh dt\label{KHMdelta3}\\
&+2\nu\mathbf{E}\int_0^T\int\int\psi\Delta\eta:(u_\delta\otimes T_h u_\delta)dx dh dt\label{KHMdelta4}\\
&+ \mathbf{E}\int_0^T\int\int p_\delta\eta:(\nabla\psi\otimes T_h u_\delta)dx dh dt - \mathbf{E}\int_0^T\int\int \psi\Div\eta\cdot(T_h p_\delta u_\delta- p_\delta T_h u_\delta) dx dh dt\label{KHMdeltap}\\
&+\mathbf{E}\int_0^T\int\int (u_\delta\cdot\nabla\psi)\left(\eta:(u_\delta\otimes T_h u_\delta)\right)dx dh dt\label{KHMdelta6}\\
&+\mathbf{E}\int_0^T \int\int \psi(x)\eta:(K_\delta Z_t\otimes T_h u_\delta + u_\delta \otimes T_h K_\delta Z_t)dx dh dt\label{KHMdelta7}\\
&+\mathbf{E}\int_0^T\int I_C:\eta dh dt.\label{KHMdeltalast}
\end{align}

We now want to take the limit $\delta \rightarrow 0$ and produce a relation for the solution $u.$ 

First of all, note that $\|u\|_{L_{\omega,x}^3}$ is finite. Since $\|u_\delta - u\|_{L_x^3} \leq C \|u\|_{L_x^3}$ and $\|u_\delta - u\|_{L_x^3}\rightarrow 0$ as $\delta \rightarrow 0,$ by the dominated convergence theorem it follows that $\|u_\delta - u\|_{L_{\omega,x}^3}\rightarrow 0 $ as $\delta \rightarrow 0.$ For convenience and clarity, we henceforth will freely use component notation and employ the Einstein summation convention. Thus, observe that for the terms on the left hand side of \eqref{KHMdelta1} we have for fixed $t$:
\begin{align}\left|\int \eta(h):\left(\Gamma_\delta(t,h)-\Gamma(t,h)\right)dh\right| &= \left| \mathbf{E}\int\int \eta^{ij}\psi\left(u_\delta^iT_h u_\delta^j - u^iT_hu^j\right)dxdh\right|\notag\\
&\leq C\|\eta\|_{L^\infty}\left(\mathbf{E}\|\psi (u_\delta^i - u^i) T_h u_\delta^j\|_{L^1} + \mathbf{E}\|T_{-h}(\psi u^i)(u_\delta^j - u^k)\|_{L^1}\right)\notag\\
&\leq C\|\eta\|_{L^\infty}\|\psi\|_{L_x^3}\|u_\delta - u\|_{L_{\omega,x}^3}\|u\|_{L_{\omega,x}^3} \rightarrow 0 \text{ as } \delta \rightarrow 0.\label{KHMlimit1}
\end{align}

Similarly, the following also hold:\noeqref{S1,S2,S3,S4,S5,S6,S7,S8,S9,S10}
\begin{equation}\mathbf{E}\int_0^T\int\int\Delta\psi\eta(h):\left(u_\delta\otimes T_hu_\delta - u\otimes T_h u\right)dxdhdt \rightarrow 0 \text{ as } \delta \rightarrow 0,\label{S1}
\end{equation}
\begin{equation}\mathbf{E}\int_0^T\int\int(\partial_{x_k}\psi\partial_{h_k}\eta(h)):\left(u_\delta\otimes T_h u_\delta - u\otimes T_h u\right)dx dh dt \rightarrow 0 \text{ as } \delta\rightarrow 0\label{S2}
\end{equation}
\begin{equation}\mathbf{E}\int_0^T\int\int\psi\Delta\eta(h):\left(u_\delta\otimes T_hu_\delta - u\otimes T_hu\right)dxdhdt\rightarrow 0 \text{ as } \delta \rightarrow 0,\label{S3}
\end{equation}
For the term in \eqref{KHMdelta2}, we only need to modify the previous argument in order to account for the fact that there are three $u$ factors in the products and therefore we have:
\begin{align}\left|\mathbf{E}\int_0^T\int\int \partial_k\psi\eta^{ij}(h)\left(u_\delta^ku_\delta^iT_hu_\delta^j-u^ku^iT_hu^j\right) dx dh \di t\right| &\leq C\|\psi\|_{L^\infty}\|\eta\|_{L^\infty}\|u\|_{L_{\omega,t,x}^3}^2\|u_\delta - u\|_{L_{\omega,t,x}^3}\\&\rightarrow 0 \text{ as } \delta\rightarrow 0.\label{S4}
\end{align}
The same argument suffices to show:
\begin{equation}\mathbf{E}\int_0^T\int\psi\partial_{h_k}\eta^{ij}(h)\left(u_\delta^iT_hu_\delta^j(u_\delta^k-T_hu_\delta^k)-u^iT_hu^j(u^k-T_hu^k)\right)dxdhdt \rightarrow 0, \text{ as } \delta \rightarrow 0.\label{S5}
\end{equation}
We now discuss the terms involving the pressure. For the first term in line \eqref{KHMdeltap}, we have:
\begin{align}
    \left|\mathbf{E}\int_0^T\int \int \eta^{ij}\partial_i\psi (p_\delta T_h u_\delta^j - pT_hu^j)dx dh dt\right| &\leq C\|\eta\|_{L^\infty}\|\psi\|_{L^\infty}\|p_\delta - p\|_{L_{\omega,t,x}^{3/2}}\|u\|_{L_{\omega,t,x}^3} \notag\\&+C\|\eta\|_{L^\infty}\|\psi\|_{L^\infty}\|p\|_{L_{\omega,t,x}^{3/2}}\|u_\delta - u\|_{L_{\omega,t,x}^3}\notag\\
    &\rightarrow 0 \text{ as } \delta \rightarrow 0,\label{S6}
\end{align}
following the same lines of reasoning as before but now exploiting \hyperref[a2]{Assumption 2} as well.
We similarly obtain:
\begin{align}
    \left|\mathbf{E}\int_0^T\int\int \psi\Div\eta\cdot (p_\delta T_h u_\delta  - p T_h u) dx dh dt\right| &\leq C\|\psi\|_{L^\infty}\|\Div\eta\|_{L^\infty}\|p_\delta - p\|_{L_{\omega,t,x}^{3/2}}\|u\|_{L_{\omega,t,x}^3} \notag\\
    &+C\|\psi\|_{L^\infty}\|\Div\eta\|_{L^\infty}\|p\|_{L_{\omega,t,x}^{3/2}}\|u_\delta-u\|_{L_{\omega,t,x}^3}\notag\\
    &\rightarrow 0 \text{ as } \delta \rightarrow 0\label{S7}
\end{align}
as above. However, the limit can be rewritten as:
\begin{align}
\mathbf{E}\int_0^T\int\int\psi\Div\eta \cdot T_h u p dx dh dt &= \mathbf{E}\int_0^T\int\int_{\ell=0}^\infty\int_{|h| = \ell} \Phi(\ell)\psi p (T_h u\cdot \hat{h})dS(h)d\ell dx dt\notag\\
&= \mathbf{E}\int_0^T\int\int_{\ell=0}^\infty \int_{|h|\leq \ell}\Phi(\ell) \psi p \Div_h T_h u dh d\ell dx dt = 0, \label{S8}
\end{align} 
where $\Phi(\ell) = \phi'(\ell) + \varphi'(\ell) + 2\ell^{-1}\varphi'(\ell)$ comes from the calculation $\Div\eta(h) = \Phi(|h|)\hat{h}$ and we have used the divergence theorem and incompressibility of $u$. Thus this term does not give any contribution in the limit.

The last pressure term is treated in the same way but does not seem to vanish in the limit, giving instead a term that will later be seen to give rise to $\frac{1}{2}$ times the pressure term in the local energy inequality:
\begin{align}
    \mathbf{E}\int_0^T\int\int \psi \Div\eta(T_h p_\delta u_\delta) dx dh dt \rightarrow \mathbf{E}\int_0^T\int\int \psi \Div\eta\cdot(T_h p u)dx dh dt \text{ as } \delta \rightarrow 0.\label{S9}
\end{align}
Notice that the limit can be integrated by parts in $h$ (which can be formally justified by approximation) to obtain: 
\begin{equation}\mathbf{E}\int_0^T\int\int \psi \Div \eta \cdot(T_h p u) dx dh dt =
\mathbf{E}\int_0^T\int\int\psi \eta : (u\otimes T_h \nabla p) dx dh dt \label{S10}
\end{equation}
As for the energy input term, we have:
\begin{align}\left|\mathbf{E}\int_0^T\int\int\psi\eta:(u_\delta\otimes T_hK_\delta Z_t + K_\delta Z_t \otimes T_h u_\delta - u\otimes T_h Z_t - Z_t\otimes T_h u) dx dh dt\right| \leq \notag\\
C\|\psi\|_{L^\infty}\|\eta\|_{L^\infty}\left(\|u_\delta - u\|_{L_{\omega,t,x}^2}\|Z_t\|_{L_{\omega,t,x}^2} + \|K_\delta Z_t - Z_t\|_{L_{\omega,t,x}^2}\|u\|_{L_{\omega,t,x}^2}\right) \rightarrow 0 \text{ as } \delta \rightarrow 0.\label{KHMlimitlast}\end{align}
The last thing to deal with is the commutator terms, however those will all be of lower order and tend to zero in the limit $\delta\rightarrow 0,$ as is shown in \cite{EG}.
This fact, the equality in the lines \eqref{KHMdelta1} through \eqref{KHMdeltalast} and the information on the limits collected between \eqref{KHMlimit1} and \eqref{KHMlimitlast} give the following balance relation:
\begin{align}&\int \eta:\Gamma(T,h)dh - \int \eta:\Gamma(0,h) dh = \mathbf{E}\int_0^T\int\int (u\otimes T_h u):\left((u\cdot\nabla\psi)\eta+\psi\partial_{h_k}\eta\delta_{h}u^k\right)dx dh dt \notag\\
&+\nu\mathbf{E}\int_0^T\int\int (u\otimes T_h u):\left[\Delta\psi\eta -2\partial_{x_k}\psi\partial_{h_k}\eta + 2 \psi\Delta\eta\right]dx dh dt\notag\\
&+ \mathbf{E}\int_0^T\int\int p\eta:(\nabla\psi\otimes T_h u)dx dh dt
- \mathbf{E}\int_0^T\int\int \psi (u\otimes\nabla(T_hp)): \eta dx dh dt\notag\\
&+\mathbf{E}\int_0^T\int\int\psi\eta:\left(u\otimes T_hZ_t+ Z_t\otimes T_h u\right)dx dh dt.\notag
\end{align}

The left hand side above vanishes, as $\Gamma$ is independent of time by \hyperref[L2.2]{Lemma 2.2}, using the stationarity of $u$. Using \hyperref[23]{Lemma 2.3} we get: $$\mathbf{E}\int\int \psi (u\otimes T_h u):\partial_{h_k}\eta\delta_hu^kdx dh = -\frac{1}{2}\int\partial_{h_k}\eta:D(h)^kdh -\frac{1}{2}\mathbf{E}\int\int\eta:T_h(u\otimes u)(\delta_h u\cdot\nabla\psi)dxdh.$$
Plugging this in the relation above and using the stationarity of $u$ to drop the integration in time completes the proof.
\end{proof}
\section{Proof of the Local 4/3 Law}\label{sec:p43}
In this section we prove \eqref{L431} and \eqref{L432} under the assumptions of their respective theorems. \hyperref[KHM]{Lemma 2.4} comes into play here, as well as in the next section. Our use of it consists in plugging in appropriate forms of isotropic tensors $\eta^{ij}$ as suggested by the statement of the local KHM relation, then switching to polar coordinates and deriving a differential equation for the structure function in the sense of distributions. After that point we can start using our assumptions to deal with the nonlinear flux terms and obtain the scaling laws we desire.

We use the following shorthand notations:
\begin{equation}
\bar{\Gamma}(\ell) = \frac{1}{4\pi}\int_{S^2}I:\Gamma(\ell \hat{n})dS(\hat{n}),
\end{equation}
\begin{equation}
\bar{Z}(\ell) = \frac{1}{4\pi}\mathbf{E}\int_{S^2}\int(u\cdot T_{\ell\hat{n}}Z_t+ T_{\ell\hat{n}}u \cdot Z_t)dS(\hat{n}),
\end{equation}
where $I$ is again the $3$ by $3$ identity matrix.
\begin{proof}[Proof of \eqref{L431}]
We start by picking our test function to be of the form $\eta^{ij}(h) = \delta^{ij}\phi(|h|).$ This in particular implies that:
$$\partial_{h_k}\eta^{ij}(h) = \delta^{ij}\frac{h_k}{|h|}\phi'(|h|).$$
Using the fact that $\bar{\Gamma}$ is spherically symmetric and $C^2$ (\hyperref[L2.2]{Lemma 2.2}) one obtains:
$$\int \Delta\eta :\Gamma(h)dh = \int \eta:\Delta\Gamma(h)dh = 4\pi \int_{\ell=0}^\infty\phi(\ell) (\ell^2\bar{\Gamma}''(\ell)+2\ell\bar{\Gamma}'(\ell)) d\ell,$$
whereby using \hyperref[KHM]{Lemma 2.4} we obtain:

\begin{align}
    &\frac{1}{2}\mathbf{E}\int_0^\infty\int_{S^2}\int\phi'(\ell)\ell^2|\delta_{\h}u|^2\delta_{\h}u\cdot\hat{n}dxdS(\hat{n})d\ell +
    \frac{1}{2}\mathbf{E}\int_0^\infty\int_{S^2}\int\phi(\ell)\ell^2|\Th u|^2\delta_{\h}u\cdot\nabla\psi dx dS(\hat{n})d\ell =\notag\\
    &2\nu\cdot 4\pi\int_0^\infty\phi(\ell)\ell^2(\bar{\Gamma}''(\ell)+\frac{2}{\ell}\bar{\Gamma}'(\ell))d\ell + 
    \nu\mathbf{E}\int_0^\infty\int_{S^2}\int\ell^2\phi(\ell)\Delta\psi (u\cdot \Th u)dx dS(\hat{n})d\ell \notag\\&- 
    2\nu\mathbf{E}\int_0^\infty\int_{S^2}\int\phi'(\ell)\ell^2(u\cdot \Th u)(\nabla\psi\cdot\hat{n})dxdS(\hat{n})d\ell\notag\\
    &+\mathbf{E}\int_0^\infty\int_{S^2}\int\phi(\ell)\ell^2p\Th u\cdot\nabla\psi dx dS(\hat{n})d\ell-
    \mathbf{E}\int_0^\infty\int_{S^2}\int\phi(\ell)\ell^2\psi T_{\ell\hat{n}}(\nabla p) \cdot u dxdS(\hat{n})d\ell\notag\\
    &+\mathbf{E}\int_0^\infty\int_{S^2}\int\phi(\ell)\ell^2(u\cdot\nabla\psi)(u\cdot \Th u)dxdS(\hat{n})d\ell
    +4\pi\int_0^\infty\phi(\ell)\ell^2\bar{Z}(\ell)d\ell.\notag
\end{align}

Set some further notation:

\begin{equation}
    \bar{F}(\ell) = \frac{1}{4\pi}\mathbf{E}\int_{S^2}\int|\Th u|^2\delta_{\h}u\cdot\nabla\psi dx dS(\hat{n}),
\end{equation}
\begin{equation}
    \bar{G}(\ell) = \frac{1}{4\pi}\mathbf{E}\int_{S^2}\int\Delta\psi(u\cdot \Th u)dx dS(\hat{n}),
\end{equation}
\begin{equation}
    \bar{Q}(\ell) = \frac{1}{4\pi}\mathbf{E}\int_{S^2}\int(u\cdot \Th u)(\nabla\psi\cdot\hat{n})dxdS(\hat{n}),
\end{equation}
\begin{equation}
	\bar{H}(\ell) = \frac{1}{4\pi}\mathbf{E}\int_{S^2}\int(u\cdot\nabla\psi)(u\cdot \Th u)dxdS(\hat{n})
\end{equation}
\begin{equation}
    \bar{P}_1(\ell) = \frac{1}{4\pi}\mathbf{E}\int_{S^2}\int p\Th u\cdot\nabla\psi dxdS(\hat{n}),
\end{equation}
\begin{equation}
    \bar{P}_2(\ell) = \frac{1}{4\pi}\mathbf{E}\int_{S^2}\int\psi u\cdot \nabla (T_{\ell\hat{n}}p)dx dS(\hat{n}).
\end{equation}
A fact that we will use is that $\bar{P}_2(0) = -\bar{P}_1(0),$ which can be seen by an integration by parts. 
We thus have the following equation in the sense of distributions:
\begin{align}-(S_0(\ell)\ell^2)' &= 4\nu (\ell^2\bar{\Gamma}'(\ell))' + 2\ell^2\bar{Z}(\ell)\\
&+2\nu\ell^2\bar{G}(\ell) + 4\nu(\ell^2\bar{Q}(\ell))'\\
&+2\ell^2\bar{P}_1(\ell) - 2 \ell^2\bar{P}_2(\ell)\\
&+2\ell^2\bar{H}(\ell) - \ell^2\bar{F}(\ell).
\end{align}
By \hyperref[L2.1]{Lemma 2.1} (resp. \hyperref[L2.2]{Lemma 2.2}) we obtain that $S_0(\ell)\ell^2\rightarrow 0$ (resp. $\bar{\Gamma}(\ell)\ell^2 \rightarrow 0$) as $\ell \rightarrow 0.$ Thus, after integrating with respect to $\ell$ and dividing by $\ell^3$, we have:
\begin{align}
-\frac{S_0(\ell)}{\ell}&=  4\frac{\nu\bar{\Gamma}'(\ell)}{\ell} + \frac{2}{\ell^3}\int_0^\ell\tau^2 \bar{Z}(\tau)d\tau \\
&+\frac{2\nu}{\ell^3}\int_0^\ell\tau^2 \bar{G}(\tau)d\tau + \frac{4\nu}{\ell} \bar{Q}(\ell)\\
&+\frac{2}{\ell^3}\int_0^\ell \tau^2 \bar{P}_1(\tau)d\tau - \frac{2}{\ell^3}\int_0^\ell\tau^2\bar{P}_2(\tau)d\tau\\
&+\frac{2}{\ell^3}\int_0^\ell \tau^2 \bar{H}(\tau)d\tau - \frac{1}{\ell^3}\int_0^\ell \tau^2\bar{F}(\tau) d\tau.
\end{align}
Now, for each fixed $\nu$ we have: 
\begin{equation}\frac{2}{\ell^3}\int_0^\ell\tau^2(\bar{P}_1(\tau)-\bar{P}_1(0))d\tau = \frac{2}{\ell^3}\frac{1}{4\pi}\mathbf{E}\int_0^\ell\int_{S^2}\int\tau^2p\delta_{\tau\hat{n}}u \cdot\nabla\psi dx dS(\hat{n})d\tau, \end{equation}
\begin{equation}\frac{2\nu}{\ell^3}\int_0^\ell \tau^2(\bar{G}(\tau)-\bar{G}(0))d\tau = \frac{2\nu}{\ell^3}\frac{1}{4\pi}\mathbf{E}\int_0^\ell\int_{S^2}\int \tau^2 \Delta\psi u\cdot\delta_{\tau\hat{n}}u dx dS(\hat{n}) d\tau, \end{equation}
\begin{equation}\frac{2}{\ell^3}\int_0^\ell\tau^2(\bar{H}(\tau)-\bar{H}(0))d\tau = \frac{2}{\ell^3}\frac{1}{4\pi}\mathbf{E}\int_0^\ell\int_{S^2}\int \tau^2(u\cdot\nabla\psi)(u\cdot\delta_{\tau\hat{n}}u)dxdS(\hat{n})d\tau,\end{equation}
\begin{equation}
\frac{2}{\ell^3}\int_0^\ell\tau^2(\bar{P}_2(\tau)-\bar{P}_2(0))d\tau = \frac{2}{\ell^3}\frac{1}{4\pi}\mathbf{E}\int_0^\ell\int_{S^2}\int\tau^2\psi u\cdot\nabla(\delta_{\tau\hat{n}}p)dxdS(\hat{n})
d\tau.\end{equation}

In what follows, we will consistently use the little-$o$ notation with the understanding that everything vanishes \textit{uniformly in} $\nu$, i.e. we write $$f(\ell) = o(g(\ell)) \text{ if and only if } \lim_{\nu\rightarrow 0}\lim_{\ell\rightarrow0}\frac{f(\ell)}{g(\ell)}=0.$$ 
For the stochastic forcing term, it is immediate by the estimates on $Z$ that:
\begin{equation}
    \frac{1}{\ell^3}\int_0^\ell\tau^2\bar{Z}(\tau)d\tau = \frac{1}{3}\bar{Z}(0) + o(1).
\label{b43}\end{equation}
We now begin to make use of \hyperref[a1]{Assumption 1}, pointing out where \eqref{WAD} is sufficient. 
By \eqref{WAD} there is $\ell_\nu \rightarrow 0$ with $\nu$ such that $\sqrt{\nu\mathbf{E}\|u\|_{L_x^2}^2} = o(\ell_\nu).$
For the $\bar{\Gamma}'(\ell)$ term, notice that:
\begin{align}
    \left|\frac{\nu\bar{\Gamma}'(\ell)}{\ell}\right| \leq \|\psi\|_{\infty} \frac{\nu}{4\pi\ell}\mathbf{E}\int_{S^2}\int |u(x)| |\nabla T_{\ell\hat{n}}u(x)|dx dS(\hat{n}) \leq C \frac{\nu}{\ell} (\mathbf{E}\|u\|_{L_x^2}^2)^{1/2}(\mathbf{E}\|\nabla u\|_{L_x^2}^2)^{1/2},\notag
\end{align}
so that:
\begin{align}\limsup_{\ell_I \rightarrow 0}\limsup_{\nu\rightarrow 0}\sup_{\ell\in(\ell_\nu,\ell_I)}\left|\frac{\nu\bar{\Gamma}(\ell)}{\ell}\right| &\leq C\limsup_{\ell_I \rightarrow 0}\limsup_{\nu\rightarrow 0}\sup_{\ell\in(\ell_\nu,\ell_I)}\frac{\nu^{1/2}(\mathbf{E}\|u\|_{L_x^2}^2)^{1/2}}{\ell_\nu}\left(\nu^{1/2}(\mathbf{E}\|\nabla u\|_{L_x^2}^2)^{1/2}\right) \notag\\
&\leq C \limsup_{\ell_I \rightarrow 0}\limsup_{\nu\rightarrow 0}\sup_{\ell\in(\ell_\nu,\ell_I)}\frac{(\nu\mathbf{E}\|u\|_{L_x^2}^2)^{1/2}}{\ell_\nu} = 0,\label{143}
\end{align}
where we have used \eqref{stb12}. Similarly, one sees that: \noeqref{143,243,343,443,543,643}
\begin{equation}
\limsup_{\ell_I\rightarrow 0}\limsup_{\nu\rightarrow 0}\sup_{\ell\in(\ell_\nu,\ell_I)}\left|\frac{\nu}{\ell}\bar{Q}(\ell)\right| \leq C\limsup_{\ell_I\rightarrow 0}\limsup_{\nu\rightarrow 0}\sup_{\ell\in(\ell_\nu,\ell_I)}\frac{\nu\mathbf{E}\|u\|_{L_x^2}^2}{\ell_\nu} = 0\label{243}
\end{equation}
As for the $\bar{G}$ term, we note:
\begin{align}\frac{\nu}{\ell^3}\int_0^\ell \tau^2\bar{G}(\tau)d\tau &= \frac{\nu}{\ell^3}\int_0^\ell \tau^2(\bar{G}(\tau)-\bar{G}(0))d\tau + \frac{\nu}{3}\bar{G}(0)\notag\\
&= \frac{\nu}{\ell^3}\frac{1}{4\pi}\mathbf{E}\int_0^\ell\int_{S^2}\int \tau^2 \Delta\psi u\cdot\delta_{\tau\hat{n}}u dx dS(\hat{n}) d\tau + \frac{\nu}{3}\bar{G}(0). \notag\end{align}
We now use the straightforward bound:
\begin{align}\left|\frac{\nu}{\ell^3}\frac{1}{4\pi}\mathbf{E}\int_0^\ell\int_{S^2}\int \tau^2 \Delta\psi u\cdot\delta_{\tau\hat{n}}u dx dS(\hat{n}) d\tau\right| &\leq C \frac{\nu}{\ell^3}\int_0^\ell\tau^2\mathbf{E
}\int_{S^2}\int |\Delta\psi||u||\delta_{\tau\hat{n}}u|dxdS(\hat{n})d\tau\notag\\
&\leq C\nu\mathbf{E}\|u\|_{L_x^2}^2,\notag
\end{align}
so we can use \eqref{WAD} to see that:
\begin{align}\limsup_{\ell_I\rightarrow0}\limsup_{\nu\rightarrow0}\sup_{\ell\in(\ell_\nu,\ell_I)}\left|\frac{\nu}{\ell^3}\int_0^\ell\tau^2(\bar{G}(\tau)-\bar{G}(0)) d\tau\right| = 0,\label{343}
\end{align}
which can also be rewritten as $\nu\ell^{-3}\int_0^\ell\tau^2\bar{G}(\tau)d\tau=\frac{\nu}{3}\bar{G}(0) + o(1).$ The reader might notice that $\nu\bar{G}(0)$ vanishes in the inviscid limit by \eqref{WAD}. The reason that we choose to keep track of it is that it will allow us, in the next proof, to use \hyperref[a3]{Assumption 3} to establish an expression for $\ell^{-1}S_0(\ell)$ in terms of the local dissipation. 

We now turn to the nonlinear terms. For the $\bar{F}$ term, we have:
\begin{align}\left|\frac{1}{\ell^3}\int_0^\ell\tau^2\bar{F}(\tau)d\tau \right| &\leq C  \frac{1}{4\pi\ell^3}\int_0^\ell\tau^2\mathbf{E}\int_{S^2}\int|\Th u|^2|\delta_{\ell\hat{n}}u||\nabla\psi| dx dS(\hat{n})d\tau \notag\\
&\leq C (\mathbf{E}\|u\|_{L_x^3}^3)^{2/3}\sup_{|h|\leq \ell}(\mathbf{E}\|\delta_hu\|_{L_x^3}^3)^{1/3},\notag
\end{align}
and therefore:
\begin{align}\limsup_{\ell_I\rightarrow 0}\limsup_{\nu\rightarrow 0}\sup_{\ell\in(\ell_\nu,\ell_I)}\left|\frac{1}{\ell^3}\int_0^\ell\tau^2\bar{F}(\tau)d\tau \right| &\leq C(\mathbf{E}\|u\|_{L_x^3}^3)^{2/3} \limsup_{\ell_I\rightarrow 0}\limsup_{\nu\rightarrow 0}\sup_{\ell\in(\ell_\nu,\ell_I)} \sup_{|h|\leq \ell}(\mathbf{E}\|\delta_hu\|_{L_x^3}^3)^{1/3}\notag\\
&\leq C(\mathbf{E}\|u\|_{L_x^3}^3)^{2/3}\limsup_{\ell\rightarrow 0}\sup_{|h|\leq \ell}(\mathbf{E}\|\delta_hu\|_{L_x^3}^3)^{1/3}=0,
\label{443}\end{align}
because of \hyperref[a1]{Assumption 1}. In other words, $\ell^{-3}\int_0^\ell\tau^2\bar{F}(\tau)d\tau = o(1).$

Notice that since
\begin{equation}\left|\frac{1}{\ell^3}\int_0^\ell\tau^2(\bar{H}(\tau)-\bar{H}(0))d\tau \right| \leq C\frac{1}{\ell^3}\int_0^\ell\tau^2\mathbf{E}\int_{S^2}\int |\nabla\psi||u|^2|\delta_{\tau\hat{n}}u|dxdS(\hat{n})d\tau,\notag\end{equation}
an argument entirely parallel to the one regarding the $\bar{F}$ term will give:
\begin{equation}\limsup_{\ell_I\rightarrow 0}\limsup_{\nu\rightarrow 0}\sup_{\ell\in(\ell_\nu,\ell_I)}\left|\frac{1}{\ell^3}\int_0^\ell\tau^2(\bar{H}(\tau)-\bar{H}(0))d\tau \right| = 0,\label{543}\end{equation}
or equivalently $\ell^{-3}\int_0^\ell\tau^2\bar{H}(\tau)d\tau = \frac{1}{3}\bar{H}(0) + o(1)$

For the $\bar{P}_1$ term, we have:
\begin{align}\left|\frac{2}{\ell^3}\frac{1}{4\pi}\mathbf{E}\int_0^\ell\int_{S^2}\int\tau^2p\delta_{\tau\hat{n}}u \cdot\nabla\psi dx dS(\hat{n})d\tau\right| &\leq C\frac{1}{\ell}\int_0^\ell \sup_{\hat{n}\in S^2} (\mathbf{E}\|\delta_{\tau\hat{n}}u\|_{L_x^3}^3)^{1/3}(\mathbf{E}\|p\|_{L_x^{3/2}}^{3/2})^{2/3}d\tau \notag\\
&\leq C (\mathbf{E}\|u\|_{L_x^{3}}^3)^{2/3}\sup_{|h|\leq \ell} (\mathbf{E}\|\delta_{h}u\|_{L_x^3}^3)^{1/3}d\tau,\notag\end{align}
where we have used \hyperref[a2]{Assumption 2}. This bound creates the same sort of situation that we had with the $\bar{F}$ term, so we can conclude that
\begin{equation}\left|
\limsup_{\ell_I\rightarrow0}\limsup_{\nu\rightarrow0}\sup_{\ell\in(\ell_\nu,\ell_I)}\frac{1}{\ell^3}\int_0^\ell\tau^2(\bar{P}_1(\tau)-\bar{P}_1(0))d\tau\right| = 0,
\label{643}\end{equation}
or equivalently $\ell^{-3}\int_0^\ell\tau^2\bar{P}_1(\tau)d\tau = \frac{1}{3}\bar{P}_1(0) + o(1).$
Finally, for the $\bar{P}_2$ term we will first integrate by parts in $x$, so that we get a situation similar to the one for $\bar{P}_1$:
\begin{equation}
\frac{1}{\ell^3}\int_0^\ell\tau^2(\bar{P}_2(\tau)-\bar{P}_2(0))d\tau = -\frac{1}{\ell^3}\frac{1}{4\pi}\mathbf{E}\int_0^\ell\int_{S^2}\int\tau^2 (u\cdot\nabla\psi)\delta_{\tau\hat{n}}pdxdS(\hat{n})
d\tau, \notag\end{equation}
which we can estimate as:

\begin{align}\left|\frac{1}{\ell^3}\int_0^\ell\tau^2(\bar{P}_2(\tau)-\bar{P}_2(0))d\tau\right| &\leq C \frac{1}{\ell}\int_0^\ell(\mathbf{E}\|u\|_{L_x^3}^3)^{1/3}\sup_{\hat{n}\in S^2}(\mathbf{E}\|\delta_{\tau\hat{n}}p\|_{L_x^{3/2}}^{3/2})^{2/3}d\tau\notag\\
&\leq C(\mathbf{E}\|u\|_{L_x^3}^3)^{1/3}\sup_{|h|\leq\ell}(\mathbf{E}\|\delta_hu\|_{L_x^3}^3)^{2/3},\notag
\end{align}
where we have again used \hyperref[a2]{Assumption 2}.
The above gives:
\begin{equation}\left|
\limsup_{\ell_I\rightarrow0}\limsup_{\nu\rightarrow0}\sup_{\ell\in(\ell_\nu,\ell_I)}\frac{1}{\ell^3}\int_0^\ell\tau^2(\bar{P}_2(\tau)-\bar{P}_2(0))d\tau\right| = 0,\label{e43}
\end{equation}
or equivalently $\ell^{-3}\int_0^\ell\tau^2\bar{P}_2(\tau)d\tau = \frac{1}{3}\bar{P}_2(0) + o(1).$

Gathering the information we have collected above in equations \eqref{b43} through \eqref{e43}, we get:
\begin{align}-\frac{S_0(\ell)}{\ell} = \frac{2}{3}\left(2\bar{P}_1(0) + \bar{H}(0) + \bar{Z}(0)+ \nu\bar{G}(0)\right) + o(1),\label{end:L43}
\end{align}
where as we have explained above the $o(1)$ vanishes as $\ell\rightarrow 0,$ uniformly in $\nu$. By the definitions of $\bar{P}_1,\bar{H},\bar{Z}$ and since $\nu\bar{G}(0)$ is also $o(1)$ by \eqref{WAD}, this completes the proof of \eqref{L431}.
\end{proof}
\begin{proof}[Proof of \eqref{L432}] We merely need to notice that the terms inside the parentheses in the right hand side of \eqref{end:L43} are precisely the ones seen in the right hand side of \eqref{LEE}.
\end{proof}

\section{Proof of the Local 4/5 Law}\label{sec:45}
Our strategy here is similar to that in the previous section. However we will also need to use the $4/3$ law itself.
\begin{proof}[Proof of \eqref{L451}] While in the proof of the local $4/3$ law our isotropic test function $\eta$ was of the form $\eta^{ij}(h) = \varphi(|h|)\delta^{ij},$ now we will take it to be of the other form suggested by the statement of the local KHM relation, namely we take:
\begin{equation}\eta^{ij}(h) = \frac{h^i h^j}{|h|^2}\varphi(|h|),\end{equation}or equivalently: \begin{equation} \eta(h) = \varphi(|h|)\hat{h}\otimes\hat{h}, \text{ where } \hat{h}=\frac{h}{|h|}.
\end{equation}
We have \begin{equation}\partial_{h_k}\left(\varphi(|h|)\hat{h}\otimes\hat{h}\right) = (\varphi'(|h|)-2|h|^{-1}\varphi(|h|))\hoh \hat{h}^k + |h|^{-1}\varphi(|h|)(e^k\otimes\hat{h}+\hat{h}\otimes e^k),\end{equation}
where $e^k$ is the unit vector in the $k$-th spatial direction\footnote{not to be confused with $e_k,$ the $k$-th eigenvector of the Stokes operator}.
In this context, \hyperref[KHM]{Lemma 2.4} gives:

\begin{gather}\int\partial_{h_k}\left(\varphi(|h|)(\hoh)\right):\left(\frac{1}{2}D^k(h)+2\nu\partial_{h_k}\Gamma(h)+2\nu\mathbf{E}\int\partial_{x_k}\psi (u\otimes T_h u) dx\right)dh \notag\\+\frac{1}{2}\mathbf{E}\int\int\varphi(|h|)(\hoh):T_h(u\otimes u)\delta_h u\cdot\nabla\psi dx dh = \mathbf{E}\int\int\varphi(|h|)(\hoh):(T_hu\otimes Z_t + u\otimes T_hZ_t)dh\notag\\
 + \mathbf{E}\int\int p\varphi(|h|)\hoh:(\nabla\psi\otimes T_hu)dx dh + \mathbf{E}\int\int\psi \varphi(|h|)(\hoh) (\nabla(T_h p)\otimes u) dx dh \notag\\
+\nu\mathbf{E}\int\int\Delta\psi\varphi(|h|) (\hoh) :(u\otimes T_h u)dx dh +\mathbf{E}\int\int u\cdot\nabla\psi \varphi(|h|)(\hoh):(u\otimes T_h u)dx dh.\notag
\end{gather}
We now wish to rewrite everything in terms of the variable $\ell$ as in the proof of the local $4/3$ law. With that in mind, we introduce the following notation:
\noeqref{45def2,45def3,45def4,45def5,45def6,45def7,45def8}
\begin{equation}\label{45def1}\tilde{\Gamma}(\ell) := \frac{1}{4\pi}\mathbf{E}\int_{S^2}\int\psi(x) (\hat{n}\cdot u) (\non:T_{\ell\hat{n}}\nabla u)dx dS(\hat{n}),\end{equation}
\begin{equation}\label{45def2}\tilde{Q}_1(\ell) := \frac{1}{4\pi}\mathbf{E}\int_{S^2}\int(\nabla\psi\cdot\hat{n})(u\cdot\hat{n})(T_{\ell\hat{n}}u\cdot\hat{n})dxdS(\hat{n}),\end{equation}
\begin{equation}\label{45def3}\tilde{Q}_2(\ell) := \frac{1}{4\pi}\mathbf{E}\int_{S^2}\int(u\cdot\nabla\psi)(T_{\ell\hat{n}}u\cdot\hat{n})dx dS(\hat{n}),\end{equation}
\begin{equation}\label{45def4}
\tilde{Q}_3(\ell) := \frac{1}{4\pi}\mathbf{E}\int_{S^2}\int(u\cdot\hat{n})(T_{\ell\hat{n}}u\cdot\nabla\psi)dxdS(\hat{n}),\end{equation}
\begin{equation}\label{45def5}\tilde{F}(\ell) := \frac{1}{4\pi}\mathbf{E}\int_{S^2}\int |T_{\ell\hat{n}}u\cdot\hat{n}|^2\delta_{\ell\hat{n}}u\cdot\nabla\psi dxdS(\hat{n}),\end{equation}
\begin{equation}\label{45def6}\tilde{Z}(\ell) := \frac{1}{4\pi}\mathbf{E}\int_{S^2}\int(\non):(T_{\ell\hat{n}}u\otimes Z_t + u\otimes T_{\ell\hat{n}}Z_t)dS(\hat{n}),\end{equation}
\begin{equation}\label{45def7}\tilde{P}_1(\ell):= \frac{1}{4\pi}\mathbf{E}\int_{S^2}\int p(\nabla\psi\cdot\hat{n})(T_{\ell\hat{n}}u\cdot\hat{n})dxdS(\hat{n}),\end{equation}
\begin{equation}\label{45def8}\tilde{G}(\ell) := \frac{1}{4\pi}\mathbf{E}\int_{S^2}\int \Delta\psi (u\cdot\hat{n})(T_{\ell\hat{n}}u\cdot\hat{n})dxdS(\hat{n}),\end{equation}
\begin{equation}\label{45def9}\tilde{H}(\ell) := \frac{1}{4\pi} \mathbf{E}\int_{S^2}\int (u\cdot\nabla\psi)(u\cdot\hat{n})(T_{\ell\hat{n}}u\cdot\hat{n})dxdS(\hat{n}).\end{equation}

Notice that in the above we have omitted an expression for the second pressure term on the right hand side of the \hyperref[KHM]{KHM} relation. The reasons for this are the following. To begin with, a term involving $T_h(\nabla p)$ cannot be estimated on the basis of \hyperref[a2]{Assumption 2}. In addition, we cannot merely integrate by parts and use the divergence free condition on $u$ in order to obtain an expression only involving $\nabla\psi$ as we did in the proof of the local $4/3$ law, since any such manipulation produces a term involving $\nabla u:(\hoh),$ which in turn cannot be estimated on the basis of \hyperref[a1]{Assumption 1}. Instead, we will recast this pressure term in a different form involving more convenient expressions. To this end, we recall the following elementary identity:
\begin{equation}\hat{h}_i\hat{h}_j = \delta^{ij}-|h|\partial_{h_i}\hat{h}_j.\end{equation}
We now use this to rewrite:
\begin{align}\mathbf{E}\int_{\mathbb{R}^3}\int \psi\varphi(|h|)(\hoh):(u\otimes\nabla(T_h p))dx dh 
=&\mathbf{E}\int_{\mathbb{R}^3}\int \psi(x)\varphi(|h|)\hat{h}_i\hat{h}_ju^i\partial_jT_hp dxdh \notag\\=& \mathbf{E}\int_{\mathbb{R}^3}\int \psi(x)\varphi(|h|)(\delta^{ij}-|h|\partial_{h_i}\hat{h}_j)u^i\partial_jT_hp dxdh \notag\\=&- \mathbf{E}\int_{\mathbb{R}^3}\int\psi(x)\varphi(|h|)|h|\partial_{h_i}\hat{h}_ju^i\partial_jT_hpdxdh \label{82}
\\  &+4\pi\int_{\ell=0}^\infty\ell^2\varphi(\ell)\bar{P}_2(\ell)d\ell.\notag\end{align}
For the term in \eqref{82}, we integrate by parts in $h$ to dispose of the pressure gradient:
\begin{align}
- \mathbf{E}\int_{\mathbb{R}^3}\int\psi(x)\varphi(|h|)|h|\partial_{h_i}\hat{h}_ju^i\partial_{h_j}T_hpdxdh &= \mathbf{E}\int_{\mathbb{R}^3}\int\partial_{h_j}(\varphi(|h|)|h|)\partial_{h_i}\hat{h}_j\psi u^iT_hpdx dh \\&+\mathbf{E}\int_{\mathbb{R}^3}\int\varphi(|h|)|h|\partial_{h_j}\partial_{h_i}\hat{h}_j\psi u^iT_h pdxdh.\label{85}
\end{align}
Notice that for the function $g(\ell):= \varphi(\ell)\ell,$ we have $\partial_{h_j}\left(g(|h|)\right)\partial_{h_i}\hat{h}_j = g'(|h|)\hat{h}_j\partial_{h_i}\hat{h}_j,$ which vanishes because $\hat{h}_j\partial_{h_i}\hat{h}_j = \frac{1}{2}\partial_{h_i}(\hat{h}_j\hat{h}_j) = \frac{1}{2}\partial_{h_i}1 = 0.$ On the other hand, for the term in \eqref{85}, notice that since $\partial_{h_i}\hat{h}_j = \partial_{h_j}\hat{h}_i,$ we get:
\begin{align}\mathbf{E}\int_{\mathbb{R}^3}\int\varphi(|h|)|h|\partial_{h_j}\partial_{h_i}\hat{h}_j\psi u^iT_h pdxdh &= \mathbf{E}\int_{\mathbb{R}^3}\int\varphi(|h|)|h|\partial_{h_j}\partial_{h_j}\hat{h}_i\psi u^iT_h p dxdh\notag\\
&=-2\mathbf{E}\int_{\mathbb{R}^3}\int\frac{\varphi(|h|)}{|h|}\hat{h}_iT_{-h}(\psi u^i)p dxdh\notag\\
&=-2\int_{\ell=0}^\infty\frac{\varphi(\ell)}{\ell}\mathbf{E}\int\int_{|h|=\ell} T_{-h}(\psi u)\cdot\hat{h}p(x)dS(|h|)dxd\ell\notag\\
&=2\int_{\ell=0}^\infty\frac{\varphi(\ell)}{\ell}\mathbf{E}\int\int_{|y|\leq \ell}T_{-y}(u\cdot\nabla\psi) p(x)dy dx d\ell\notag\\
&=8\pi\int_{\ell=0}^\infty\frac{\varphi(\ell)}{\ell}\tilde{P}_2(\ell)d\ell,
\end{align}
where we denote: 
\begin{equation}\label{45P2}\tilde{P}_2(\ell):=\frac{1}{4\pi}\mathbf{E}\int\int_{|y|\leq \ell}(u\cdot\nabla\psi)T_{y}pdydx.\end{equation}

As in the proof of the local $4/3$ law, the expressions \eqref{45def1} through \eqref{45def9} and \eqref{45P2} will now be used to rewrite the integrals of the KHM relation in spherical coordinates. As an example, we only show how $\tilde{\Gamma}$ appears, the other changes of variables being simpler:
\begin{align}\int_{\mathbb{R}^3}\partial_{h_k}(\phi(|h|)\hoh):2\nu\partial_{h_k}\Gamma(h)dh &= 2\nu \mathbf{E}\int_{\mathbb{R}^3}\int\psi(x)((\phi'(|h|)-2|h|^{-1}\phi(|h|))(u^i\hat{h}_i)(\partial_{h_k}T_hu^j\hat{h}_j\hat{h}_k)dx dh\notag
\\&+2\nu\mathbf{E}\int_{\mathbb{R}^3}|h|^{-1}\phi(|h|)(\delta^{i,k}\hat{h}_j+\hat{h}_i\delta^{j,k})\int\psi(x)u^i\partial_kT_hu^jdxdh\label{visc}\\
&=8\pi\nu\int_{\ell=0}^\infty\ell^2(\phi'(\ell)-2\ell^{-1}\phi(\ell))\tilde{\Gamma}(\ell)d\ell\label{422},
\end{align}
the expression in line \eqref{visc} being equal to $0$ because:
\begin{align}\mathbf{E}\int_{\mathbb{R}^3}|h|^{-1}\phi(|h|)(\delta^{i,k}\hat{h}_j+\hat{h}_i\delta^{j,k})\int\psi(x)u^i\partial_kT_hu^jdxdh =\notag\\ \mathbf{E}\int_{\ell=0}^\infty\ell\phi(\ell)\int \int_{S^2}\psi(x)(u^i\partial_iT_{\ell\hat{n}}u^j\hat{n}_j + \hat{n}_i\partial_jT_{\ell\hat{n}}u^j)dS(\hat{n})dxd\ell=\notag\\
\mathbf{E}\int_{\ell=0}^\infty \ell\phi(\ell)\int\int_{|y|\leq 1} \psi(x)u^i\partial_{y_j}T_{\ell y}u^jdydxd\ell = 0,\notag
\end{align}
where we have used the divergence theorem and the incompressibility of $u$.\\ Using that $\ell^2\phi'(\ell) -2\ell\phi(\ell) = \ell^4(\ell^{-2}\phi(\ell))',$ we rewrite our KHM relation as an equality of integrals in $\ell$:

\begin{align}&-\frac{1}{2}\int_{\ell=0}^\infty\ell^{-2}\phi(\ell)(\ell^4S_{||}(\ell))'d\ell + \int_{\ell=0}^\infty\ell S_0(\ell)d\ell -2\nu\int_{\ell=0}^\infty\ell^{-2}\phi(\ell)(\ell^4\tilde{\Gamma}(\ell))'d\ell \\
+ 2\nu&\int_{\ell=0}^\infty\left(-\ell^{-2}\phi(\ell)(\ell^4\tilde{Q}_1(\ell))' +\ell\phi(\ell)(\tilde{Q}_2(\ell)+\tilde{Q}_3(\ell))\right)d\ell= \\
&\int_{\ell=0}^\infty\ell^2\phi(\ell)\left(\tilde{Z}(\ell)+\tilde{P}_1(\ell)+\bar{P}_2(\ell) + \nu\tilde{G}(\ell) + \tilde{H}(\ell)-\frac{1}{2}\tilde{F}(\ell)\right)d\ell+2\int_{\ell=0}^\infty
\frac{\phi(\ell)}{\ell}\tilde{P}_2(\ell)d\ell.
\end{align}
In the sense of distributions on $(0,\infty)$, this translates to the following ODE:
\begin{align}-\frac{1}{2}(\ell^4S_{||}(\ell))' + \ell^3 S_0(\ell) -&2\nu(\ell^4\tilde{\Gamma}(\ell))'
+ 2\nu\left(-(\ell^4\tilde{Q}_1(\ell))' +\ell^3(\tilde{Q}_2(\ell)+\tilde{Q}_3(\ell))\right)= \\
&\ell^4\left(\tilde{Z}(\ell)+\tilde{P}_1(\ell)+\bar{P}_2(\ell) + \nu\tilde{G}(\ell) + \tilde{H}(\ell)-\frac{1}{2}\tilde{F}(\ell)\right)+ 2\ell\tilde{P}_2(\ell).\end{align}
As in the proof of the $4/3$ law, we now integrate from $0$ to $\ell$ and multiply by $2\ell^{-5}$ in order to obtain an expression for the quantity of interest $\ell^{-1}S_{||}(\ell):$
\begin{align}&-\frac{S_{||}(\ell)}{\ell} + \frac{2}{\ell^5}\int_0^\ell\tau^3 S_0(\tau)d\tau -\frac{4\nu\tilde{\Gamma}(\ell)}{\ell} -\frac{4\nu\tilde{Q}_1(\ell)}{\ell} + \frac{4\nu}{\ell^5}\int_0^\ell\tau^3(\tilde{Q}_2(\tau)+\tilde{Q}_3(\tau))d\tau = \\
&\frac{2}{\ell^5}\int_0^\ell\tau^4\left(\tilde{Z}(\tau)+\tilde{P}_1(\tau)+\bar{P}_2(\tau) + \nu\tilde{G}(\tau) + \tilde{H}(\tau)\right)d\tau+\frac{1}{\ell^5}\int_0^\ell\tau^4\tilde{F}(\tau)d\tau+\frac{4}{\ell^5}\int_0^\ell\tau\tilde{P}_2(\tau)d\tau.
\end{align}
Again, the $\ell\rightarrow 0$ boundary contributions are not present because $$\lim_{\ell\rightarrow0}\ell^4S_{||}(\ell) = \lim_{\ell\rightarrow0}\ell^4\tilde{\Gamma}(\ell) = \lim_{\ell\rightarrow0}\ell^4\tilde{Q}_1(\ell) =0,$$
by \hyperref[L2.1]{Lemma 2.1}, \hyperref[L2.2]{Lemma 2.2} and
the bound $\mathbf{E}\|u\|_{L_x^2}^2 <\infty$ used for $\tilde{Q}_1.$
We now start using our assumptions. Once again, in what follows the little-$o$ notation denotes \textit{uniform in} $\nu$ vanishing at the limit $\ell\rightarrow0.$
For the $\tilde{Z}$ term, it follows by the estimates for the Ornstein-Uhlenbeck SPDE that:
\begin{equation}
\frac{1}{\ell^5}\int_0^\ell\tau^4\tilde{Z}(\tau)d\tau = \frac{1}{5}\tilde{Z}(0) + o(1). \label{L45first}
\end{equation}
For the $\tilde{\Gamma}$ term, we have:
\begin{equation}
    \left|\frac{\nu\tilde{\Gamma}(\ell)}{\ell}\right| \leq C\frac{\nu}{\ell}\mathbf{E}\int_{S^2}\int |u| |\Th \nabla u|dx dS(\hat{n}) \leq C\frac{\nu}{\ell}(\mathbf{E}\|u\|_{L_x^2}^2)^{1/2}(\mathbf{E}\|\nabla u\|_{L_x^2}^2)^{1/2},\notag
\end{equation}
so we pick our scales $\ell_\nu$ such that $(\nu\mathbf{E}\|u\|_{L_x^2}^2)^{1/2} = o(\ell_\nu),$ which is possible according to \eqref{WAD}, and we use \eqref{stb12} to obtain:\noeqref{145,245,345,445,545,645,745}
\begin{equation}
\limsup_{\nu \rightarrow 0} \sup_{\ell\in(\ell_\nu,\ell_I)}\left|\frac{\nu}{\ell}\tilde{\Gamma}(\ell)\right| \leq C \limsup_{\nu\rightarrow 0} \frac{(\nu\mathbf{E}\|u\|_{L_x^2}^2)^{1/2}}{\ell_\nu} = 0.
\label{145}\end{equation}
Similarly, observe that:
\begin{equation}\left|\frac{\nu}{\ell}\tilde{Q}_1(\ell)\right| \leq C\frac{\nu}{\ell}\mathbf{E}\int_{S^2}\int|u||\Th u|dxdS(\hat{n}) \leq \frac{\nu \mathbf{E}\|u\|_{L_x^2}^2}{\ell},\end{equation}
so by the same procedure as above, we obtain:
\begin{equation}\limsup_{\nu \rightarrow 0} \sup_{\ell\in(\ell_\nu,\ell_I)}\left|\frac{\nu}{\ell}\tilde{Q}_1(\ell)\right|\leq C \limsup_{\nu\rightarrow 0} \frac{\nu}{\ell_\nu}\mathbf{E}\|u\|_{L_x^2}^2 \leq C\limsup_{\nu\rightarrow0}\frac{(\nu\mathbf{E}\|u\|_{L_x^2}^2)^{1/2}}{\ell_\nu} = 0.\label{245}\end{equation}
Next, for $\mu=2,3,$ we have:
\begin{equation}
    \left|\frac{\nu}{\ell^5}\int_0^\ell\tau^3\tilde{Q}_\mu(\tau)d\tau\right|\leq C\frac{\nu}{\ell} \mathbf{E}\|u\|_{L_x^2}^2,\notag
\end{equation}
so once again in the same way as before \eqref{WAD} gives:
\begin{equation}\limsup_{\nu\rightarrow 0}\sup_{\ell\in(\ell_\nu,\ell_I)}\left|\frac{\nu}{\ell^5}\int_0^\ell\tau^3\tilde{Q}_\mu(\tau)d\tau\right| \leq C\limsup_{\nu\rightarrow 0} \frac{\nu\mathbf{E}\|u\|_{L_x^2}^2}{\ell_\nu} = 0.\label{345}
\end{equation}
For the term involving $\tilde{G},$ we have:
\begin{equation}
\left|\frac{\nu}{\ell^5}\int_0^\ell\tau^4(\tilde{G}(\tau)-\tilde{G}(0))d\tau\right| \leq C\nu \mathbf{E}\|u\|_{L_x^2}^2, \notag
\end{equation}
so that by \eqref{WAD}:
\begin{align}\limsup_{\nu\rightarrow 0}\sup_{\ell\in(\ell_\nu,\ell_I)}\left|\frac{\nu}{\ell^5}\int_0^\ell\tau^4(\tilde{G}(\tau)-\tilde{G}(0))d\tau\right| = 0\label{445}.\end{align} In other words, we have $\nu\ell^{-5}\int_0^\ell\tau^4\tilde{G}(\tau)d\tau = \frac{\nu}{5}\tilde{G}(0) + o(1).$
At this point our treatment of the viscous terms is complete. We turn to the nonlinear terms.

For the $\tilde{F}$ term, we have:
\begin{equation}|\tilde{F}(\tau)| \leq C (\mathbf{E}\|u\|_{L_x^3}^3)^{2/3}(\mathbf{E}\|\delta_{\tau\hat{n}}\|_{L_x^3}^3)^{1/3},\notag \end{equation}
which vanishes uniformly in $\nu$ when $\tau \rightarrow 0$ by \hyperref[a1]{Assumption 1}, implying that:
\begin{align}\limsup_{\ell_I\rightarrow 0}\limsup_{\nu\rightarrow 0}\sup_{\ell\in(\ell_\nu,\ell_I)}\left|\ell^{-5}\int_0^\ell\tau^4\tilde{F}(\tau)d\tau\right| &\leq C(\mathbf{E}\|u\|_{L_x^3}^3)^{2/3} \limsup_{\ell_I\rightarrow 0}\limsup_{\nu\rightarrow 0}\sup_{|h|\leq \ell_I} (\mathbf{E}\|\delta_{h}\|_{L_x^3}^3)^{1/3}\notag \\ 
&= 0.\label{545}
\end{align}
In an entirely analogous manner, i.e. using exactly the same bounds as for $\tilde{F},$ one obtains:
\begin{equation}\limsup_{\ell_I\rightarrow 0}\limsup_{\nu\rightarrow0}\sup_{\ell\in(\ell_\nu,\ell_I)}\left|\frac{1}{\ell^5}\int_0^\ell\tau^4(\tilde{H}(\tau)-\tilde{H}(0))d\tau \right| = 0\label{645},
\end{equation}
or in other words $\ell^{-5}\int_0^\ell\tau^4\tilde{G}(\tau) = \frac{1}{5}\tilde{G}(0) + o(1).$
For the pressure term involving $\tilde{P}_1,$ we have:
\begin{equation}|\tilde{P}_1(\tau)-\tilde{P}_1(0)| \leq C (\mathbf{E}\|p\|_{L_x^{3/2}}^{3/2})^{2/3}\sup_{\hat{n}\in S^2} (\mathbf{E}\|\delta_{\tau\hat{n}}u\|_{L_x^3}^3)^{1/3} \leq C\sup_{\hat{n}\in S^2} (\mathbf{E}\|\delta_{\tau\hat{n}}u\|_{L_x^3}^3)^{1/3},\notag
\end{equation}
where we have used \hyperref[a2]{Assumption 2} on the boundedness of the pressure. The above implies:
\begin{align}&\limsup_{\ell_I\rightarrow 0}\limsup_{\nu\rightarrow 0} \sup_{\ell\in(\ell_\nu,\ell_I)} \left|\ell^{-5}\int_0^\ell\tau^4(\tilde{P}_1(\tau)-\tilde{P}_1(0))d\tau\right| \leq\notag\\ &C\limsup_{\ell_I \rightarrow 0}\limsup_{\nu\rightarrow 0}\sup_{\ell\in(\ell_\nu,\ell_I)}\ell^{-1}\int_0^\ell\sup_{\hat{n}\in S^2}(\mathbf{E}\|\delta_{\tau\hat{n}}u\|_{L_x^3}^3)^{1/3}d\tau \leq\notag\\
&C\limsup_{\ell_I \rightarrow 0}\limsup_{\nu\rightarrow 0}\sup_{|h|\leq \ell_I}(\mathbf{E}\|\delta_{h}u\|_{L_x^3}^3)^{1/3} = 0,\label{745}
\end{align}
again by \hyperref[a1]{Assumption 1}. This can also be restated as $\ell^{-5}\int_0^\ell\tau^4\tilde{P}_1(\tau)d\tau = \frac{1}{5}\tilde{P}_1(0) + o(1).$
Dealing with the $\bar{P}_2$ term follows an absolutely identical argument to the one in the proof of the local $4/3$ law, save for the different powers of $\ell$ and $\tau$ that appear. We simply remark that as for the previous terms, we again have $\ell^{-5}\int_0^\ell\tau^4\bar{P}_2(\tau)d\tau = \frac{1}{5}\bar{P}_2(0) + o(1).$ It remains to deal with the $\tilde{P}_2$ term. We have:
\begin{align}\frac{1}{\ell^5}\int_0^\ell\tau\tilde{P}_2(\tau)d\tau &= \frac{1}{\ell^5}\int_0^\ell\frac{\tau}{4\pi}\mathbf{E}\int\int_{|y|\leq \tau}(u\cdot\nabla\psi)T_{y}pdydxd\tau\notag\\
&= \frac{1}{\ell^5}\int_0^\ell\tau^4\left( \frac{1}{4\pi\tau^3}\mathbf{E}\int\int_{|y|\leq \tau}(u\cdot\nabla\psi)T_y p dydx - \frac{1}{3}\mathbf{E}\int (u\cdot\nabla\psi)pdx\right)d\tau\notag\\
&+\frac{1}{3\ell^5}\int_0^\ell\tau^4\mathbf{E}\int(u\cdot\nabla\psi)pdx d\tau\notag\\
&= \frac{1}{\ell^5}\int_0^\ell\tau^4\left(\frac{1}{4\pi\tau^3}\mathbf{E}\int\int_{|y|\leq\tau}(u\cdot\nabla\psi)\delta_ypdydx\right)d\tau + \frac{1}{15}\bar{P}_1(0),
\end{align}
and also:
\begin{align}\limsup_{\ell_I\rightarrow0}\limsup_{\nu\rightarrow 0}\sup_{\ell\in(\ell_\nu,\ell_I)}&\left|\ell^{-5}\int_0^\ell \tau^4\left(\frac{1}{4\pi\tau^3}\mathbf{E}\int\int_{|y|\leq\tau}(u\cdot\nabla\psi)\delta_y pdydx\right)d\tau\right| \notag\\&\leq C (\mathbf{E}\|u\|_{L_x^3}^3)^{1/3}\limsup_{\ell_I\rightarrow 0}\limsup_{\nu\rightarrow0}\sup_{|y|\leq \ell_I}(\mathbf{E}\|\delta_yp\|_{L_x^{3/2}}^{3/2})^{2/3}\notag\\
&= 0,
\label{L45last}\end{align}
where we have used both \hyperref[a1]{Assumption 1} and \hyperref[a2]{Assumption 2}. In this case similarly to before we have: $\frac{1}{\ell^5}\int_0^\ell\tau\tilde{P}_2(\tau)d\tau = \frac{1}{15}\bar{P}_1(0) + o(1).$
Collecting the information from \eqref{L45first} through \eqref{L45last}, we have:
\begin{align}\label{alm}&-\frac{S_{||}(\ell)}{\ell} + \frac{2}{\ell^5}\int_0^\ell\tau^3 S_0(\tau)d\tau = \frac{2}{5}\left(\tilde{Z}(0)+\tilde{P}_1(0)+\bar{P}_2(0) + \nu\tilde{G}(0) + \tilde{H}(0)\right)+\frac{4}{15}\bar{P}_1(0) + o(1). 
\end{align}
Noticing that $\int_{S^2}\non dS(\hat{n})=\frac{4\pi}{3}I,$ we have:
\begin{equation}\tilde{Z}(0) = \frac{1}{3}\bar{Z}(0), \tilde{P}_1(0) = \frac{1}{3}\bar{P}_1(0), \tilde{G}(0) = \frac{1}{3}\bar{G}(0), \tilde{H}(0)=\frac{1}{3}\bar{H}(0),\notag
\end{equation}
so combining \eqref{alm} with \eqref{L431}, we obtain:
\begin{align}-\frac{S_{||}(\ell)}{\ell} &= \frac{2}{15}\left(\bar{Z}(0)+5\bar{P}_1(0)+3\bar{P}_2(0) + \bar{H}(0) + \nu\bar{G}(0)\right)\notag\\& +\frac{4}{3\ell^5}\int_0^\ell\tau^4\left(\bar{H}(0)+2\bar{P}_1(0) + 2\bar{Z}(0)+ \nu\bar{G}(0)+ o(1)\right)d\tau  + o(1)\notag\\
&=\frac{2}{5}\left(\bar{Z}(0)+2\bar{P}_1(0)+\bar{H}(0)+\nu\bar{G}(0)\right)+o(1),\label{e451}
\end{align}
and since \eqref{WAD} implies that $\nu\bar{G}(0)$ vanishes as $\nu\rightarrow0$,
we see that \eqref{e451} furnishes \eqref{L451}.
\end{proof}
\begin{proof}[Proof of \eqref{L452}] The proof is the same as the proof of \eqref{L432}, namely, by noticing that the terms in the parentheses in the right hand side of \eqref{e451} are exactly the right hand side of \eqref{LEE}.
\end{proof}

\appendix
\section{Existence of stationary martingale solutions for OU driven Navier-Stokes}\label{Appendix}
We sketch the construction of stationary martingale solutions to \eqref{eqn:NSE}, which closely follows the approach in \cite{FG} for the case of white-in-time forcing. 

Given $k\in \mathbb{N},$ we write Galerkin approximations to \eqref{eqn:NSE}:
\begin{equation}\tag{GNSE}\label{GNSE}
    \begin{cases}\partial_t u^k + P^kB(u^k) + \nu Au^k = Z_t^k\\
    dZ_t^k = -LZ_t^k + P^kdW_t\\
    \end{cases}
\end{equation}
where $P^k$ is the projection onto the subspace of $H$ generated by $(e_j)_{j=1}^k,$ $B(u) = P(\Div(u\otimes u)),$ $Au = -P\Delta u$ and $P$ is the Leray projection operator.

This system is a coupling of an Ornstein-Uhlenbeck SDE with a finite dimensional (random) differential equation with Lipschitz coefficients. 

We can completely adapt the derivation of the estimates found for the process $u_t^k$ in \cite{FG} to the joint process $(u_t^k,Z_t^k)$ in our setting. To begin with, notice that by Ito's formula it follows that for $p\geq 2$:
\begin{align} d\|Z_t^k\|_H^p \leq p \|Z_t^k\|_H^{p-2}\left<Z_t^k,dZ_t^k\right>+ \frac{p(p-1)}{2}\|Z_t^k\|_H^{p-2}\ep dt\implies\\
d\|Z_t^k\|_H^p + p \|Z_t^k\|_H^{p-2}\left<Z_t^k,LZ_t^k\right>\leq p\|Z_t^k\|_H^{p-2}\left<Z_t^k,QdW_t\right> + \frac{p(p-1)}{2}\|Z_t^k\|_H^{p-2}\ep dt\implies\\
d\|Z_t^k\|_H^p + pc_1\|Z_t^k\|_H^{p-2}\|Z_t^k\|_V^2\leq p\|Z_t^k\|_H^{p-2}\left<Z_t^k,QdW_t\right> + C\left(\|Z_t^k\|_H^pdt + \ep^{p/2}\right) dt. \label{PP}
\end{align}
where $\left<\cdot,\cdot\right>$ is the inner product of $H$ and the constant $C$ depends only on $p.$
Hence, by taking expectations we obtain:
\begin{equation}
\mathbf{E}\|Z_t^k\|_H^p + pc_1\mathbf{E}\int_0^t\|Z_s^k\|_H^{p-2}\|Z_s^k\|_V^2 ds \leq \mathbf{E}\|Z_0^k\|_H^p +  C\int_0^t\mathbf{E}\|Z_s^k\|_H^pds + C\ep^{p/2} t
\end{equation}
which implies by Gronwall's inequality that for every time interval $[0,T]$ there exists a finite positive constant which we again denote by $C,$ depending only on $p$ and $T,$ such that
\begin{equation}
\mathbf{E}\|Z_t^k\|_H^p \leq C
\end{equation}
for all $t$ in $[0,T].$

This in turn implies the existence of another positive constant (again denoted by $C$) depending on $p,T$ and $c_1,$ such that: \begin{equation}\mathbf{E}\int_0^T\|Z_t^k\|_H^{p-2}\|Z_t^k\|_V^2 \leq C < \infty.
\end{equation}
Plugging in $p=2,$ we obtain the fundamental estimate:
\begin{equation}\mathbf{E}\int_0^T\|Z_t^k\|_V^2 \leq C < \infty.
\end{equation}
Note that none of the bounds above depend on $k,$ because we have assumed that $Z_0^k = 0$ for all $k\geq 1.$ The only difference between the above and the estimates carried out in \cite{FG} is that we need to assume coercivity of the bilinear form $\left<Lu,v\right>$ in the space $V$ (ensuring which is the role of the constant $c_1$) whereas for $A$ this is immediate since one just takes $\left<Av,v\right> = \|v\|_{V}^2$ by definition. 

Next, we apply the Burkholder-Davis-Gundy inequality to get that for an absolute constant $\tilde{C}$:
\begin{align}\mathbf{E}\sup_{0\leq t \leq T}\left|\int_0^t \|Z_s^k\|_{H}^{p-2}\left<Z_s^k,QdW_s\right>\right| &\leq \tilde{C} \mathbf{E}\left(\int_0^T \|Z_s^k\|_H^{2p-2}\ep ds \right)^{\frac{1}{2}}\\
&\leq \tilde{C}\mathbf{E}\left[\sup_{0\leq t\leq T}\|Z_s^k\|_{H}^{\frac{p}{2}}\left(\int_0^T\|Z_s^k\|_H^{p-2}\ep ds\right)^{\frac{1}{2}}\right]\\
&\leq \frac{1}{2p}\mathbf{E}\sup_{0\leq t\leq T}\|Z_s^k\|_H^p + \frac{p}{2}\tilde{C}^2\mathbf{E}\int_0^T\|Z_s^k\|_H^{p-2}\ep ds\\
&\leq \frac{1}{2p}\mathbf{E}\sup_{0\leq t\leq T}\|Z_s^k\|_H^p + \frac{pC\tilde{C}^2}{2}\mathbf{E}\int_0^T\|Z_s^k\|_H^p ds + \frac{pC\tilde{C}^2}{2}\ep^{p/2}t,
\end{align}
where the constant $C$ is the one that appears in line \eqref{PP}.

From \eqref{PP} and the inequality just derived, it follows that for a new constant $C$:
\begin{align}
\mathbf{E}\sup_{0\leq t\leq T}\|Z_t^k\|_H^p \leq \mathbf{E}\|Z_0^k\|_H^p+ \frac{1}{2}\mathbf{E}\sup_{0\leq t\leq T}\|Z_t^k\|_H^p + C\mathbf{E}\int_0^T\|Z_s^k\|_H^p ds + C\ep^{p/2}t,
\end{align}
hence by Gronwall's we obtain that there exists a  positive constant $C$ depending on $p$ and $T$ such that:
\begin{equation}\mathbf{E}\sup_{0\leq t\leq T}\|Z_t^k\|_H^p \leq C<\infty.\end{equation}
The same analysis can be carried out for the Galerkin equation governing $u_t^k$ using classical rather than stochastic calculus, i.e. this case is more elementary than the one treated in the Appendix in \cite{FG}.

We therefore arrive at estimates of the form:

\begin{align}
    \mathbf{E}\int_0^T\|(u_t^k,Z_t^k)\|_{V\times V}^2dt \leq C < \infty,\label{gstb1}\\
    \mathbf{E}\sup_{0\leq t\leq T}\|(u_t^k,Z_t^k)\|_{H\times H}^p \leq C < \infty.\label{gstb2}
\end{align}
The rest of the work done in \cite{FG} can be repeated in exactly the same way. One uses the Skorokhod embedding theorem and ends up with a stochastic basis $(\tilde{\Omega},\tilde{F},\{\tilde{F}_t\}_{t\geq 0}, \tilde{P},\tilde{W}_t),$ $\tilde{W}_t$ being a $Q$-Wiener process on which we will also have a process $(\tilde{u}_t,\tilde{Z}_t)$ satisfying \eqref{eqn:NSE} in the sense of \hyperref[D1]{Definition 1.2}. 

Note that a priori it is only shown that $\tilde{Z}_t$ satisfies the equation 
$$d\tilde{Z}_t = -L\tilde{Z}_t dt + d\tilde{W}_t$$
in the weak sense. However, it is easy to see that under conditions specified in \cite[Chapter~5]{DPZ} this weak solution is in fact strong and in the stationary case corresponds to an invariant measure of the Ornstein-Uhlenbeck SPDE above.
Note that the estimates \eqref{gstb1} and \eqref{gstb2} are used to derive stationary martingale solutions satisfying \eqref{stb1} and \eqref{stb2}.\\

\subsection*{Acknowledgements}
The author would like to thank his advisor, Jacob Bedrossian, for suggesting the problem addressed in this paper and providing guidance. The author would also like to thank the anonymous referees for their helpful insights, which significantly improved this work. This work was partially supported by Jacob Bedrossian's NSF CAREER grant DMS-1552826.

\footnotesize{\bibliography{bibliographyfinal}}

\end{document}